\documentclass[12pt]{article}
\usepackage{amsmath,amsfonts,amsthm,amssymb, graphicx,epsfig}
\usepackage[usenames,dvipsnames]{xcolor}
\usepackage[left=1in,top=1in,right=1in]{geometry}
\newcommand{\Z}{\mathbb{Z}}

\newcommand{\out}{\mathrm{out}}

\newtheorem{them}{Theorem}

\newtheorem{thm}{Theorem}[section]
\newtheorem{lem}[thm]{Lemma}
\newtheorem{prop}[thm]{Proposition}

\newtheorem{cor}[thm]{Corollary}
\newtheorem{remark}{Remark}
\newtheorem{conjecture}{Conjecture}
\newtheorem{question}{Question}

\numberwithin{equation}{section}

\begin{document}

\title{Bigeodesics in first-passage percolation}
\date{\today}
\author{Michael Damron\thanks{The research of M. D. is supported by NSF grant DMS-1419230 and an NSF CAREER grant.} \\ \small Georgia Tech and Indiana University \and Jack Hanson\thanks{The research of J. H. is supported by an AMS-Simons travel grant and NSF grant DMS-1612921.} \\ \small Georgia Tech and CUNY}
\maketitle

\begin{abstract}
In first-passage percolation, we place i.i.d. continuous weights at the edges of $\mathbb{Z}^2$ and consider the weighted graph metric. A distance-minimizing path between points $x$ and $y$ is called a geodesic, and a bigeodesic is a doubly-infinite path whose segments are geodesics. It is a famous conjecture that almost surely, there are no bigeodesics. In the `90s, Licea-Newman showed that, under a curvature assumption on the ``asymptotic shape,'' all infinite geodesics have an asymptotic direction, and there is a full measure set $D \subset [0,2\pi)$ such that for any $\theta \in D$, there are no bigeodesics with one end directed in direction $\theta$. In this paper, we show that there are no bigeodesics with one end directed in any deterministic direction, assuming the shape boundary is differentiable. This rules out existence of ground state pairs for the related disordered ferromagnet whose interface has a deterministic direction. Furthermore, it resolves the Benjamini-Kalai-Schramm ``midpoint problem'' \cite[p. 1976]{BKS} under the extra assumption that the limit shape boundary is differentiable.
\end{abstract}

\section{Introduction}
Consider first-passage percolation (FPP) on $\mathbb{Z}^2$. The model is defined as follows. Let $(t_e)$ be a collection of non-negative random variables, one assigned to each edge in $\mathcal{E}^2$, the nearest-neighbor edges of $\mathbb{Z}^2$. The weighted graph (pseudo-)metric is defined as 
\[
T(x,y) = \inf_{\gamma : x \to y} T(\gamma),
\]
where $\gamma$ is any lattice path from $x$ to $y$; that is, a sequence $(x=x_0,e_0,x_1, e_1, \ldots, e_{n-1},x_n=y)$ of vertices and edges such that for $k=0, \ldots, n-1$, the edge $e_k \in \mathcal{E}^2$ equals $\{x_k, x_{k+1}\}$, and $T(\gamma)$ is the passage time of $\gamma$:
\[
T(\gamma) = \sum_{e \in \gamma} t_e.
\]

A minimizing path from $x$ to $y$ in the above definition of $T$ is called a geodesic. Infinite paths all of whose finite segments are geodesics are called infinite geodesics. A bigeodesic is an infinite geodesic which is doubly infinite; that is, its vertices are indexed by $\mathbb{Z}$. The following is a well-known conjecture, attributed by Kesten as a question of Furstenberg \cite[Remark~9.22]{aspects}.
\begin{conjecture}\label{conj: bigeodesics}
Suppose $(t_e)$ is i.i.d. with continuous distribution. Almost surely there are no bigeodesics in dimension two.
\end{conjecture}

In this article, we provide new progress toward this conjecture. Below is special case of Theorem~\ref{thm: limits}, which considers more general translation invariant distributions and weaker limit shape assumptions. In the statement, a path with vertices $x_1, x_2, \ldots$ has direction $\theta$ if $\|x_n\|_1 \to \infty$ and $\mathrm{arg}~x_n \to \theta$. The ``limiting shape'' is defined below in Section~\ref{sec: assumptions}.
\begin{thm}\label{thm: main_fake_thm}
Suppose the collection $(t_e)$ is i.i.d. with continuous distribution and $\mathbb{E}t_e < \infty$. Assume that the model's limit shape $\mathcal{B}$ has a differentiable boundary. Given $\theta \in [0,2\pi)$, almost surely there are no bigeodesics with one end having direction $\theta$.
\end{thm}
Note that there are two main differences between this result and the above conjecture. First, we assume that $\partial \mathcal{B}$ is differentiable. This is expected to be true for all dimensions and all continuous distributions but, nonetheless, is another well-known conjecture. In Section~\ref{sec: background}, we explain how in our general setting, differentiability may be necessary for many of our results. Second, our statement is about a deterministic direction, whereas the conjecture is about all directions simultaneously. Regardless, the most recent progress on this conjecture was due to Wehr-Woo \cite{WW98} in '98 and Licea-Newman \cite{LN96} in '96. In Section~\ref{sec: background}, we will present these results and discuss the degree to which ours improve on them. In contrast to Licea-Newman's methods which use curvature and concentration, our analysis is based on Busemann functions, which are tools introduced from metric space geometry to FPP by Hoffman in two influential papers in '05 and '08 \cite{Hoffman05, Hoffman08}. (Busemann-type limits were also considered previously by Newman \cite{N95} and later by Pimentel \cite{P06}.)

\section{Main results}

\subsection{Assumptions}\label{sec: assumptions}
We will assume either {\bf A1'} or {\bf A2'} below. They are conditions from \cite{DH} on the probability space $(\Omega, \Sigma, \mathbb{P})$, where $\Omega = [0,\infty)^{\mathcal{E}^2}$ and $\Sigma$ is the product Borel sigma-algebra.
\begin{enumerate}
\item[{\bf A1'}] $\mathbb{P}$ is a product measure whose common distribution is continuous and satisfies 
\[
\mathbb{E}\left(\min_{1 \leq i \leq 4} t_{e_i}\right)^2 < \infty,
\] 
where $e_1, \ldots, e_4$ are the edges touching the origin.
\item[{\bf A2'}] $\mathbb{P}$ satisfies the conditions of Hoffman \cite{Hoffman08}:
\begin{enumerate}
\item $\mathbb{P}$ is ergodic with respect to translations of $\mathbb{Z}^2$,
\item $\mathbb{P}$ has all the symmetries of $\mathbb{Z}^2$,
\item $\mathbb{E}t_e^{2+\epsilon}<\infty$ for some $\epsilon>0$, and
\item the limit shape for $\mathbb{P}$ is bounded.
\end{enumerate}
Furthermore, we assume that $\mathbb{P}$ satisfies a version of unique passage times: 
\begin{enumerate}
\item[(e)] for edge sequences $e_1, \ldots, e_n$ and $f_1, \ldots, f_m$ such that at least one $e_i$ is not equal to any of the $f_i$'s, one has $\sum_i t_{e_i} \neq \sum_i t_{f_i}$ almost surely. 
\end{enumerate}
Last, we assume the upward finite-energy condition of \cite{DH}: writing $(t_e) = (t_e, \check t)$ as the edge-weight $t_e$ along with all other edge-weights $\check t = (t_f)_{f \neq e}$, one has
\begin{enumerate}
\item[(f)] $\mathbb{P}(t_e \geq \lambda \mid \check t) > 0$ almost surely
\end{enumerate}
whenever $\lambda > 0$ is such that $\mathbb{P}(t_e \geq \lambda)>0$.
\end{enumerate}
The last condition means that the supremum of the support of $t_e$ cannot decrease if we condition on the other weights. 

The limit shape referred to above is the set $\mathcal{B} \subset \mathbb{R}^2$ given by $\mathcal{B} = \{z : g(z) \leq 1\}$, and $g$ is the asymptotic norm for the FPP model (see \cite[Theorem~2.1]{survey}) given, for each $z \in \mathbb{R}^2$, by
\[
g(z) = \lim_n T(0,nz)/n \text{ almost surely and in }L^1,
\]
where we have extended the passage time function to all of $\mathbb{R}^2$, setting $T(z_1,z_2) = T(z_1',z_2')$ if $z_i'$ is the unique lattice point with $z_i \in z_i' + [0,1)^2$. $\mathcal{B}$ is also characterized by the ``shape theorem,'' (see \cite[Theorem~2.6]{survey} and \cite{boivin}) which says that given $\epsilon>0$, one has
\[
\mathbb{P}\left( (1-\epsilon)\mathcal{B} \subset B(t)/t \subset (1+\epsilon)\mathcal{B} \text{ for all large }t \right) = 1,
\]
where $B(t)/t$ is the set $\{z/t : z \in B(t)\}$. There is no simple condition known to guarantee that the limit shape under condition {\bf A2'} is bounded. {\bf A1'} and {\bf A2'} are sufficient to ensure that between each $x,y \in \mathbb{Z}^2$, there is a unique geodesic $\Gamma(x,y)$ (see \cite[Section~4.1]{survey}).

\subsection{Main FPP results}\label{sec: main_results}
We say that a path with vertices $x_1, x_2, \ldots$ is directed in a sector $S \subset [0,2\pi)$ if $\|x_n\|_1 \to \infty$ and all the limit points of $\{\mathrm{arg}~ x_n : n \geq 1\}$ lie in $S$. (Here, $0$ and $2\pi$ are identified.) Let $\mathcal{B}$ be the limit shape and define
\[
w_\theta = (\cos \theta,\sin \theta) \text{ and } v_\theta = w_\theta / g(w_\theta) \text{ for } \theta \in [0,2\pi).
\]
Next, for some $\theta \in [\pi/4,\pi/2]$,
\begin{equation}\label{eq: diff_assumption_1}
\text{assume }\partial \mathcal{B} \text{ is differentiable at } v_\theta.
\end{equation}
(By symmetry of $\mathcal{B}$, we may assume that $\theta$ is in this interval.) Let $P$ be the tangent line at $v_\theta$ and note that $P$ cannot be a vertical line. Define $S$ to be the sector of angles $\phi$ such that $v_\phi \in P$. If $\theta_1$ is the minimal angle in $S$ and $\theta_2$ is the maximal angle in $S$ (the endpoints of $S$), then we will furthermore
\begin{equation}\label{eq: diff_assumption_2}
\text{assume }\partial \mathcal{B} \text{ is differentiable at } v_{\theta_1} \text{ and } v_{\theta_2}.
\end{equation}
(Similar definitions are made for general $\theta \in [0,2\pi)$.) Note that by symmetry considerations and differentability,
\[
0<  \theta_1 \leq \theta_2 < \pi.
\]
It was pointed out to us by a referee that the arguments of this paper go through if \eqref{eq: diff_assumption_2} is replaced by the following: $v_{\theta_i}$ is a limit of extreme points of $\mathcal{B}$ for each $i$. 

\begin{them}\label{thm: limits}
Assume {\bf A1'} or {\bf A2'}. Further, for $\theta \in [0,2\pi)$, assume \eqref{eq: diff_assumption_1} and \eqref{eq: diff_assumption_2}. The following hold with probability one.
\begin{enumerate}
\item For each $x \in \mathbb{Z}^2$, there is an infinite geodesic $\Gamma_x$ that is directed in $S$ such that for any (possibly random) sequence $(x_n)$ directed in $S$,
\[
\Gamma_x = \lim_n \Gamma(x,x_n).
\]
\item For each $x,y \in \mathbb{Z}^2$, the geodesics $\Gamma_x$ and $\Gamma_y$ coalesce.
\item There are no bigeodesics with one end directed in $S$.
\end{enumerate}
\end{them}

In the first item above, we say that a sequence of paths $(\Gamma_n)$ converges to a path $\Gamma$ if for each $k$, the first $k$ steps of $\Gamma_n$ are eventually equal to the first $k$ steps of $\Gamma$. In the second item, ``$\Gamma_x$ and $\Gamma_y$ coalesce'' means that $\#(\Gamma_x \Delta \Gamma_y) < \infty$.

\begin{remark}
Apart from being in any fixed direction, the third item above is stronger than that of Licea-Newman, stated as the second half of part 1 in Theorem~\ref{thm: LN} below. Their theorem rules out bigeodesics with both ends in fixed directions (outside of an exceptional set), whereas ours rules out bigeodesics with an end in one fixed direction.
\end{remark}

\begin{remark}\label{rem: directions}
Since bigeodesics with fixed directions cannot exist, one should ask if infinite geodesics are even required to have directions. (This is why, although part 1 of Theorem~\ref{thm: LN} below does not require a curvature assumption, it is not useful without part 2, which requires a curvature assumption and asserts that geodesics actually have directions.) One can show using planarity and the results of \cite{DH} that if $\partial \mathcal{B}$ is differentiable and either {\bf A1'} or {\bf A2'} hold, then the following statements are true with probability one:
\begin{enumerate}
\item for all $\theta$, there is an infinite geodesic starting from 0 directed in $S_\theta$ and
\item every infinite geodesic is directed in $S_\theta$ for some $\theta$.
\end{enumerate}
\end{remark}

\begin{remark}
The above theorem implies that for such $\theta$, all infinite geodesics that are directed in $S$ coalesce. This in turn has consequences for competing growth models, and we mention one example here. Given initial sites $x,y \in \mathbb{Z}^2$, an infection at $x$ colonizes all sites $z$ with $T(x,z) < T(y,z)$ (and similarly for $y$). The sets of sites infected by $x$ or $y$ are each connected and have union equal to $\mathbb{Z}^2$, so they are separated by an interface in the dual lattice. One can show that if this interface is doubly infinite (that is, both sets are infinite) and one end is directed in $S$, then there are disjoint infinite geodesics started from $x$ and $y$ directed in $S$ and this has zero probability.
\end{remark}

\begin{them}\label{thm: busemann}
Assume {\bf A1'} or {\bf A2'}. Further, for $\theta \in [0,2\pi)$, assume \eqref{eq: diff_assumption_1} and \eqref{eq: diff_assumption_2}. With probability one, for each $x,y \in \mathbb{Z}^2$ and (possibly random) sequence $(x_n)$ directed in $S$, the limit
\[
B(x,y) = \lim_n \left[ T(x,x_n) - T(y,x_n) \right]
\]
exists. Furthermore, letting $\rho$ be the unique vector in $\mathbb{R}^2$ such that $\{r \in \mathbb{R}^2 : r \cdot \rho = 1\}$ is the tangent line to $\mathcal{B}$ in direction $\theta$, one has:
\begin{enumerate}
\item $\mathbb{E}B(0,x) = \rho \cdot x$ for $x \in \mathbb{Z}^2$.
\item For each $\epsilon>0$, the set of $x \in \mathbb{Z}^2$ such that $|B(0,x) - \rho \cdot x| > \epsilon \|x\|_1$ is almost surely finite.
\end{enumerate}
\end{them}

The limit $B$ is sometimes called a Busemann function. See \cite[Section~5]{survey}.

\subsection{Connection to the disordered ferromagnet}

Our main theorem on bigeodesics has implications for the ground states of disordered ferromagnetic spin models. The typical example is the disordered Ising ferromagnet, which we define below in detail only in two dimensions. Consider the dual lattice defined by
\[
\left( \mathbb{Z}_*^2, \mathcal{E}_*^2 \right) = \left( \mathbb{Z}^2, \mathcal{E}^2 \right) + \left( \frac{1}{2}, \frac{1}{2} \right).
\]
A spin configuration $\sigma$ is an element $(\sigma_x)_{x \in \mathbb{Z}_*^2} \in \{+1,-1\}^{\mathbb{Z}_*^2}$. Let $(J_{x,y})_{\{x,y\} \in \mathcal{E}_*^2}$ have joint distribution $\mu$ which is translation-ergodic and with $\mu(J_{x,y} > 0)=1$. The (random) energy of $\sigma$ relative to the couplings $(J_{x,y})$ is defined as 
\[
H_S(\sigma) = - \sum_{\{x,y\} \in \mathcal{E}_*^2: x \in S} J_{x,y} \sigma_x \sigma_y,
\]
where $S$ is a finite subset of $\mathbb{Z}^2_*$. A ground state is defined as a configuration $\sigma$ such that whenever $\hat \sigma$ agrees with $\sigma$ everywhere except on some finite set,
\[
H_S(\sigma) \leq H_S(\hat \sigma) \text{ for all finite } S \subset \mathbb{Z}_*^2.
\]

A fundamental open problem in the study of such spin systems is to determine the number of ground states for a given $(J_{x,y})$. It is conjectured that under some critical dimension (see \cite{FLN} and \cite{Tasaki} for predictions for scaling exponents and upper critical dimension), there are two almost surely (all plus or all minus), and this should include $d=2$. There are only partial rigorous results at this point to support the conjecture, and these come from analysis of FPP geodesics. To see the connection, we can define an FPP model based on the spin system by setting $t_e = J_{x,y}$, where $e$ is the edge in the primal lattice which is dual to $\{x,y\}$. Then, as shown in \cite[Propositions~1.1-1.2]{N97}, there are nonconstant ground states in the spin model if and only if the induced FPP model has bigeodesics.

To argue this, note that if there is a nonconstant ground state $\sigma$, the symmetric difference (set of vertices where spins differ) between $\sigma$ and the $+1$ ground state $\sigma_+$ cannot contain any finite components. To see why, if such a component $S$ exists, then by the ground state property, one would have $H_S(\sigma) = H_S(\sigma_+)$. However, the spin interactions in both states in the interior of $S$ are equal, and they differ only on the boundary. This implies the sum of energy terms for bonds on the boundary must be zero in both states, which is a contradiction if the coupling distribution is continuous. Therefore any nonconstant ground state must have a two-sided (and circuitless) infinite (primal lattice) path of edges dual to bonds $\{x,y\}$ with $\sigma_x = -\sigma_y$. We call such a path an interface. This interface can be seen to be a bigeodesic for the induced FPP model.

The results of Licea-Newman (Theorem~\ref{thm: LN} below) therefore rule out existence of nonconstant ground states with interface having both ends directed in the set $\mathcal{D}$. Our Theorem~\ref{thm: main_fake_thm} therefore shows that under a natural differentiability assumption, there can be no nonconstant ground states with an interface directed in any deterministic direction.

\subsection{BKS midpoint problem}

The following problem arose from the Benjamini-Kalai-Schramm work on sublinear variance in Bernoulli FPP, and has become known as a ``missing lemma."
\begin{question}
Is it true that
\begin{equation}\label{eq: BKS_convergence}
\mathbb{P}(\lfloor n/2 \rfloor e_1 \text{ is in a geodesic from }0 \text{ to } ne_1) \to 0?
\end{equation}
\end{question}
The authors wanted to verify this to control the influence of edge-weights for the random variable $T = T(0,ne_1)$. Specifically, it would help them to use an inequality of Talagrand that roughly gives a logarithmic improvement to Efron-Stein variance bounds in settings where each underlying variable has small influence on the variable in question. They were not able to solve this question, and got around it by using an averaged version of $T$ instead.

This problem is still not solved, but our theorems give a conditional positive answer given the assumption that the boundary of the limiting shape is differentiable. A similar result would hold with $e_1$ replaced by any deterministic vector.
\begin{thm}
Assume {\bf A1'} or {\bf A2'}. Further, for $\theta = 0$, assume \eqref{eq: diff_assumption_1} and \eqref{eq: diff_assumption_2}. Then \eqref{eq: BKS_convergence} holds.
\end{thm}
\begin{proof}
Assume that $\limsup_n \mathbb{P}(\lfloor n/2 \rfloor e_1 \text{ is in a geodesic from }0 \text{ to } ne_1) > 0$. Then by translation invariance,
\[
\mathbb{P}(0 \text{ is in a geodesic from } \lfloor -n/2 \rfloor e_1 \text{ to } \lceil n/2 \rceil e_1 \text{ for infinitely many } n) > 0.
\]
Let $A$ be this event intersected with the probability one event from Theorem~\ref{thm: limits}.

On $A$, almost surely the limit of finite geodesics from 0 to $\lceil n/2 \rceil e_1$ exists and is directed in the sector $S$ given by the angles of contact of the unique tangent line to the limit shape in direction $\theta=0$ with the shape. If we let $\Gamma$ be any subsequential limit near zero of geodesics from $\lfloor -n/2 \rfloor e_1$ to $\lceil n/2 \rceil e_1$, then $\Gamma$ is a bigeodesic with one end directed in $S$. This is a contradiction to item 3 of Theorem~\ref{thm: limits}.
\end{proof}

\section{Background and sketch of proofs}

\subsection{Previous results on geodesics}\label{sec: background}

Work on infinite geodesics in FPP began with Wehr \cite{Wehr} and Wehr-Woo \cite{WW98} in the '90s. The first work showed that there are either zero or infinitely many bigeodesics, and the second showed that there are no bigeodesics confined to the upper half-plane. After this came the work of Licea-Newman, to which we compare our results. Their theorems have now been used in a large number of related models, with essentially no improvement (see \cite{BCK14, CT12, CP11, FP05} for a few). 
\begin{thm}[\cite{LN96, N95}]\label{thm: LN}
Suppose the distribution of $(t_e)$ is i.i.d. and each $t_e$ has continuous distribution.
\begin{enumerate}
\item There is a deterministic set $\mathcal{D} \subset [0,2\pi)$ whose complement has Lebesgue measure zero such that the following holds. For fixed $\theta \in \mathcal{D}$,
\[
\mathbb{P}(\text{there are two disjoint infinite geodesics in direction }\theta)=0.
\]
Furthermore, for each $\theta, \theta' \in \mathcal{D}$, almost surely, there is no bigeodesic with one end directed in direction $\theta$ and the other directed in direction $\theta'$.
\item Assume $\mathbb{E}e^{\alpha t_e}<\infty$ for some $\alpha>0$. Further, assume that the limit shape $\mathcal{B}$ is uniformly curved. Then almost surely, every infinite geodesic has an asymptotic direction and
\[
\mathbb{P}(\text{for all }\theta, \text{ there is an infinite geodesic in direction }\theta)=1.
\]
\end{enumerate}
\end{thm}
\noindent
As mentioned in Remark~\ref{rem: directions}, although item 1 does not require a curvature assumption, it does not give much information without item 2, which does need curvature. The set $\mathcal{D}$ was defined indirectly, and so there is no useful characterization of it. Shortly after the publication of the above result, Zerner \cite[Theorem~1.5]{N97} showed that $D^c$ can be taken to be countable.

We aim to improve the above results by (a) allowing distributions on $(t_e)$ that are simply translation-ergodic and (b) showing that $\mathcal{D}$ can be taken to be the entire set $[0,2\pi)$. In even framing such results, we are faced with some complications. The main issue comes from the theorem of H\"aggstr\"om-Meester \cite[Theorem~1.3]{HM}, which states that given a convex, compact set in $\mathbb{R}^d$ which has the symmetries of $\mathbb{Z}^d$, there is a translation-ergodic model of FPP on $\mathbb{Z}^d$ with this set as its limit shape. In particular, there are two-dimensional models of FPP with limit shapes that are polygons. The theorems of Damron-Hanson \cite{DH} imply that for each face of some such polygonal limit shapes, there is a geodesic asymptotically directed in the sector corresponding to this face. It is reasonable to believe that one can construct models in which the only infinite geodesics are these, and they are directed merely in sectors, but not in directions $\theta$ (they wander across the sector). In fact, behavior for geodesics in some stationary models that is quite different from that predicted in the i.i.d. case has already been displayed; see \cite{BD02} for an example with exactly one bigeodesic. So for such models, item 2 above would be false, and item 1 would not give any information.

So we are led to consider directedness in sectors and we might hope that the following analogue of the Licea-Newman theorem holds: for each sector $S$, there is zero probability that there are two disjoint infinite geodesics directed in $S$. However, this might also be false, without extra assumptions. The reason is that it is reasonable that there exist other translation-ergodic FPP models whose limit shapes are polygons, and which have infinite geodesics directed toward the corners of the polygon (since these are ``fast'' directions). In such a case, each face will have at least two infinite geodesics directed in its associated sector. The solution to this is to assume that on each face of the limit shape, one has differentiability of the boundary at the two endpoints. This is precisely our assumption for our main theorems.

In the general translation-ergodic case, as mentioned above, Damron-Hanson \cite{DH} showed existence of infinite geodesics that are directed in sectors corresponding to sides of the limit shape; that is, in sectors of the form $S_\theta$. However, they did not need a differentiability assumption at the endpoints of the sector. They also showed forms of coalescence of these directed geodesics, building coalescing trees of them which have no backward paths. Those results did not address uniqueness. For example, it was possible that multiple disjoint geodesic trees existed, all directed in the same sector. The main contribution of our work is to complete the picture in the differentiable case, and this differentiability assumption, as discussed earlier, may indeed be necessary.

One can ask if it is possible to show a uniqueness statement simultaneously in all directions. In other words, can one show that under a differentiability assumption, almost surely, all geodesics directed in a sector coalesce, for all sectors simultaneously? Due to Remark~\ref{rem: directions}, this would be much closer to proving Conjecture~\ref{conj: bigeodesics}. Unfortunately, this is almost certainly false, as it is expected that there is a random, countable set of directions in which uniqueness does not hold. It is in fact not difficult to show existence of these directions, given our results, by considering competition interfaces. (See, for example, \cite[Theorem~2.6]{GRS2}, where this is done for last-passage percolation.) The conclusion is that to rule out bigeodesics in all directions at once, it is likely that more ideas are needed than arguing through uniqueness of one-sided infinite geodesics.

\subsection{Geodesics in related models}

After the papers of Newman and coauthors in the '90s, Howard and Newman introduced in \cite{HN97} a rotationally-invariant model called Euclidean FPP. This model is on a graph whose vertices are the points of a Poisson point process and the limit shape is therefore a ball. Using this rotational symmetry, they were able to prove \cite{HN01} non-existence of bigeodesics and uniqueness of infinite geodesics in deterministic directions. The reason is that one can still show that there is a deterministic set $D \subset [0,2\pi)$ as in Theorem~\ref{thm: LN}, but in the Euclidean model, it is a rotationally invariant set whose complement has measure zero. This implies that $D=[0,2\pi)$. Apart from \cite{DH}, we know of no other progress on directional geodesics in undirected models.

Much more is known in some related directed models. In the exactly solvable last-passage percolation model, one puts exponential or geometric weights on the vertices of the first quadrant of $\mathbb{Z}^2$ and considers the maximal passage time of directed paths. The limiting shape boundary for this model is given by an explicit formula, and is uniformly curved. Using this fact and the techniques of Licea-Newman, it was shown by Ferrari-Pimentel \cite{FP05} that there are uncountably many infinite geodesics, one in each direction. From this, one can show absence of bigeodesics in deterministic directions. Further work was done by Coupier in \cite{C11}.

In 2014, a theory of infinite geodesics for last-passage percolation (LPP) with general weights was given by Georgiou-Rassoul-Agha-Sepp\"al\"ainen \cite{GRS, GRS2} that parallels the one developed by Damron-Hanson in \cite{DH} for FPP. Using directedness of paths, they were able to go further than in \cite{DH}, proving uniqueness of infinite geodesics and absence of bigeodesics in deterministic directions. The main tool they used (which was missing in FPP) is a monotonicity for Busemann functions, which comes from a ``paths crossing'' trick initially due to Alm \cite{A98} and Alm and Wierman \cite{AW99} in the late '90s. (We note that monotonicity properties of Busemann functions were also derived by Cator-Pimentel in 2011 \cite{CP11} in an exactly solvable model.) The non-directedness is precisely the main issue that we must deal with in this paper -- the majority of our arguments serve to order geodesics and give a useful definition of a left-most and right-most geodesic.

\subsection{Sketch of proofs}

We give the idea of the proof of Theorem~\ref{thm: limits}. Theorem~\ref{thm: busemann} follows from standard arguments. Given $\theta \in [\pi/4,\pi/2]$ such that the limit shape boundary is differentiable at $v_\theta$, the point in direction $\theta$, we define the sector $S_\theta$ as in the last section: it is the sector of angles of contact with $\mathcal{B}$ of the tangent line to $\mathcal{B}$ in direction $\theta$. The main idea of the proof is to make a definition of a \emph{left-most and right-most} infinite geodesic directed in $S_\theta$ from each point. From a point $x$, we label these geodesics by $\Gamma_x^L$ and $\Gamma_x^R$. For the proof, our definition must satisfy the following properties for $* = L,R$:
\begin{enumerate}
\item $\Gamma_0^*$ is asymptotically directed in $S_\theta$.
\item If $y \in \Gamma_0^*$, then $\Gamma_y^*$ is the segment of $\Gamma_0^*$ from $y$ onward.
\item For distinct $x,y$, the paths $\Gamma_x^*$ and $\Gamma_y^*$ coalesce.
\end{enumerate}
Establishing these properties is the main difficulty in the argument, and this is really where the undirectedness of the model causes problems. The construction and properties of extremal geodesics is done in Section~\ref{sec: ordering}.

Once we have extremal geodesics, we can define Busemann functions for them. For $x,y \in \mathbb{Z}^2$ and $*=L,R$, we define
\[
B^*(x,y) = \lim_n \left[ T(x,x_n) - T(y,x_n) \right],
\]
where $x_1, x_2, \ldots$ is the sequence of vertices on $\Gamma_0^*$. Property 3 above implies that we could also make this definition using $\Gamma_z^*$ for any other $z$, and get the same limiting function. Using this property and the ergodic theorem, we can show that
\[
B^*(0,x) = f^*(x) + o(\|x\|_1) \text{ as } x \to \infty
\]
for some linear functional $f^*$.

The next step is to relate to extremal upper half-plane geodesics, which we also construct in Section~\ref{sec: ordering}. 
These geodesics $\Gamma_{x,H}^*$ can be shown (in Proposition~\ref{prop: equality}) to be equal to $\Gamma_x^*$ for a positive density of $x$ on the boundary of the half-plane. Using this fact, we can show that the half-plane Busemann functions $B_H^*$ have the same asymptotic behavior as the full-plane $B^*$.

For the half-plane Busemann functions, we can prove a type of monotonicity (similar to the one known in LPP) in Section~\ref{sec: equal_geodesics}. Namely, if we define
\[
\Delta_H(x,y) = B_H^L(x,y) - B_H^R(x,y),
\]
then one has $\Delta_H(0,e_1) \leq 0$ almost surely. This comes from the ``paths crossing'' argument. However, if $\Gamma_{0,H}^L \neq \Gamma_{0,H}^R$ with positive probability, uniqueness of passage times implies in Proposition~\ref{prop: bigger} that $\mathbb{E}\Delta_H(0,e_1) < 0$ and by the ergodic theorem,
\[
\liminf_n \Delta_H(0,ne_1)/n < 0.
\]
(This part of the argument is analogous to the LPP case \cite[Theorem~2.1(iii)]{GRS2}.) Because $B_H^*(x) = f_H^*(x) + o(\|x\|_1)$ for some linear functional $f_H^*$, we deduce that
\[
f_H^L \neq f_H^R.
\]
On the other hand, using our differentiability assumption, we derive in Proposition~\ref{prop: delta_0} that in fact $f_H^L = f_H^R$, giving a contradiction. Therefore $\Gamma_{0,H}^L = \Gamma_{0,H}^R$ almost surely, and in Proposition~\ref{prop: equality}, we argue the same for the full-plane geodesics. Since the extremal geodesics coincide, we deduce uniqueness; that is, all infinite geodesics directed in $S_\theta$ coalesce.

To move from uniqueness to absence of bigeodesics, we consider the union of all infinite geodesics that are directed in $S_\theta$. This union forms a tree whose vertex set is $\mathbb{Z}^2$ and whose edge set consists of all geodesics $\Gamma_x^*$ for $x \in \mathbb{Z}^2$ and $*=L$ or $R$. We appeal to a mass-transport type result established in Damron-Hanson \cite{DH} to deduce that any such tree cannot have infinite backward paths. In other words, no infinite geodesic directed in $S_\theta$ is a subpath of a bigeodesic with one end directed in $S_\theta$, and this completes the proof.

\section{Ordering geodesics}\label{sec: ordering}

\subsection{Topological preliminaries}

From this point on, by symmetry we will take $\theta \in [\pi/4,\pi/2]$ and recall our assumptions that
\begin{enumerate}
\item $\partial \mathcal{B}$ is differentiable at $v_\theta$ and
\item $\partial \mathcal{B}$ is differentiable at $v_{\theta_1}$ and $v_{\theta_2}$, where $\theta_1$ and $\theta_2$ are the endpoints of the sector of angles of contact of $\mathcal{B}$ with the unique tangent line to $\mathcal{B}$ at $v_\theta$.
\end{enumerate}
Recall that by symmetry and differentiability,
\[
0 < \theta_1 \leq \theta_2 < \pi.
\]
This ensures that each infinite geodesic directed in $S$ only intersects $L_0 = \{(x,0) : x \in \mathbb{Z}\}$ finitely often, and therefore eventually remains on one side of it.

We will need the following result, which follows from \cite[Corollary~1.3]{DH}.
\begin{prop}\label{prop: trapping_tools}
There exist sequences of sectors $(\Theta_n^L), \, (\Theta_n^R)$ in $(0,\pi)$ such that, almost surely, there exist sequences $(\Gamma_n^L)$ and $(\Gamma_n^R)$ of infinite geodesics in $\mathbb{Z}^2$ starting from 0, with each $\Gamma_n^*$ directed in $\Theta_n^*$ for $* = L, R$. Moreover,
\begin{enumerate}
\item $\inf \Theta_{n+1}^R > \sup \Theta_n^R$ and $\sup\Theta_{n+1}^L < \inf \Theta_n^L$ for all $n$ and
\item $\inf \Theta_n^R \to \theta_1$ and $\sup \Theta_n^L \to \theta_2$.
\end{enumerate}
\end{prop}
\begin{proof}
In \cite{DH}, it is shown that if $v,w$ are extreme points of $\mathcal{B}$ and $\Theta$ is a sector of angles between them, there is an infinite geodesic directed in $\Theta$. In our case, $\partial \mathcal{B}$ is differentiable at each $v_{\theta_i}$, the endpoints of the sector $S$, so we can find sequences of extreme points converging to each $v_{\theta_i}$ from outside $S$. So we can apply the result of \cite{DH} in sectors converging to each $\theta_i$ as in the statement.
\end{proof}

It is important to note that the conditions stated in the previous proposition indeed comprise an event. In other words, given a sector $\Theta$, the set of passage-time configurations for which there is an infinite geodesic from 0 directed in $\Theta$ is measurable. This follows from arguments similar to those in \cite[Appendix~A]{N97}. One caveat, however, is that the geodesics $\Gamma_n^L$ and $\Gamma_n^R$ constructed in \cite{DH} are on a larger probability space, and the arguments therein do not directly allow them to be defined in a $(t_e)$-measurable manner. This potential lack of measurability is not important for the following arguments, as we will always fix an outcome $\omega$ and choose the geodesics arbitrarily.

In what follows, we will make use of the first-passage model on the half-plane $\mathbb{H}$, whose vertices form the set
\[
V_H = \{(x,y) \in \mathbb{Z}^2 : y \geq 0\}
\]
and whose edges form the set $E_H$ of nearest-neighbor edges. Let $T_H = T_H(x,y)$ be the half-plane passage time for $x,y \in V_H$ with $\Gamma_H(x,y)$ the unique half-plane geodesic from $x$ to $y$. We extend the various notions of directedness to $\mathbb{H}$ in the obvious way.

\begin{prop}\label{prop: trapping_tools_hp}
The statement of Proposition \ref{prop: trapping_tools} holds on $\mathbb{H}$.
\end{prop}
\begin{proof}
  Fix a choice of $n$ and recall the sequences of geodesics and sectors from Proposition \ref{prop: trapping_tools}. We will show that there is a sector $\Theta_n^{H,L}$ lying between $\Theta_{n}^L$ and $\Theta_{n+1}^L$ and a half-plane geodesic $\Gamma_n^{H,L}$ starting from $0$ directed in $\Theta_n^{H,L}$. Specifically, we will take $\Theta _n^{H,L} = [\inf \Theta_{n+1}^L, \sup \Theta_n^L]$. This (combined with an identical argument for $\Gamma_n^{H,R}$) suffices to establish the proposition.

\begin{figure}[h]
\setlength{\unitlength}{.5in}
\begin{picture}(10,3)(-5.5,2.5)
\linethickness{3pt}
\put(-4,2.5){\line(1,0){10}}
\linethickness{1pt}
\put(-2.4,3.2){$\Gamma^1$}
\put(-2.1,2.2){$x$}
\put(-2,2.5){\line(0,1){1}}
\put(-2,3.5){\line(1,0){2}}
\put(0,3.5){\line(0,1){.5}}
\put(0,4){\line(1,0){1}}
\put(1,4){\line(0,1){.5}}
\put(1,4.5){\line(1,0){1}}
\put(2,4.5){\line(0,1){1}}
\put(2,5.5){\line(1,0){.5}}
\put(2.5,5.5){\line(0,1){.5}}
\put(1.4,2.2){$y$}
\put(4,3){$\Gamma^2$}
\put(1.5,2.5){\line(0,1){.5}}
\put(1.5,3){\line(1,0){2}}
\put(3.5,3){\line(0,1){.5}}
\put(3.5,3.5){\line(1,0){.5}}
\put(4,3.5){\line(0,1){.5}}
\put(4,4){\line(1,0){1}}
\put(5,4){\line(0,1){.5}}
\put(5,4.5){\line(1,0){1}}
\put(-.2,2.2){$0$}
\linethickness{2pt}
\put(0,2.5){\line(0,1){.5}}
\put(0,3){\line(-1,0){1}}
\put(-1,3){\line(0,1){2}}
\put(-1,5){\circle*{.2}}
\put(-1.3,5){$z$}
\put(-1,5){\line(1,0){4}}
\put(0,5.2){$\Gamma_H(0,mv_\phi)$}
\put(3.2,5){$[mv_\phi]$}
\put(3,5){\circle*{.2}}
\linethickness{1pt}
\put(-1,5){\line(0,1){1}}
\put(-1.3,5.6){$\Gamma$}
\end{picture}
\caption{Illustration of the argument of Proposition~\ref{prop: trapping_tools_hp}. The half-plane infinite geodesics $\Gamma^1$ and $\Gamma^2$ are directed in $\Theta_n^L$ and $\Theta_{n+1}^L$. The region between them in the upper half-plane is denoted $R$. $\Gamma$ is an infinite geodesic that stays in the upper half-plane and touches $z$, which is a vertex that is not in $R$ or its boundary. Since $z$ is also in $\Gamma_H(0,mv_\phi)$ (in bold), and this geodesic must cross back into $R$ to reach $mv_\phi$.}
\label{fig: fig_1}
\end{figure}
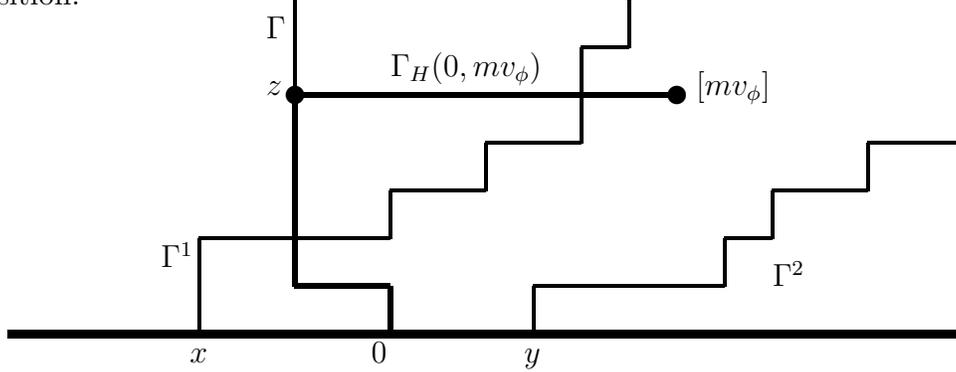

See Figure~\ref{fig: fig_1} for a depiction of the following argument. 
Since $\Gamma_n^L$ has a last intersection with $L_0$, we see that some vertex on $L_0$ has positive probability to have a half-plane geodesic directed in $\Theta_n^L$. (Again, by measurability arguments in \cite[Appendix~A]{N97}, this condition defines an event.) By the ergodic theorem, there almost surely exist $x$ to the left of $0$ and $y$ to the right of $0$ and $\Gamma^1$, $\Gamma^2$ infinite geodesics in $\mathbb{H}$ starting at $x$ and $y$ (respectively) and directed in $\Theta_n^L$ and $\Theta_{n+1}^L$ (respectively).

Now look on the event that (a) there exist $x$ and $y$ as above and (b) all distinct finite paths have distinct passage times. Choose an angle $\phi$ in $(\sup \Theta_{n+1}^L, \inf \Theta_{n}^L)$, and let $\Gamma$ be a subsequential limit of $(\Gamma_H(0, nv_\phi))_{n \geq 1}$. We claim that $\Gamma$ is directed in $\Theta_n^{H,L}$. Suppose to the contrary that there were some $\epsilon > 0$ and infinitely many vertices $z \in \Gamma$ with 
\[\arg z \notin [\inf \Theta_{n+1}^L - \epsilon, \sup \Theta_n^L + \epsilon]\ .\] 
Now, let $P$ be the simple path formed by the union of $\Gamma^1, \Gamma^2$, and the segment of the $e_1$-axis between $x$ and $y$, and let $R$ be the component of $\mathbb{R}^2 \setminus P$ which contains all but finitely many $\{nv_\phi\}$.

By the definition of $\Gamma$, there is a vertex $z \in (\mathbb{R}^2 \setminus (R \cup P))$ and a value of $m$ such that (a vertex within distance 1 of) $m v_\phi$ lies in $R$ and $z \in \Gamma_H(0, m v_\phi)$. In particular, the segment of $\Gamma(0, m v_\phi)$ from $z$ to $m v_\phi$ connects $R$ to the other component of $\mathbb{R}^2 \setminus P$ and so must intersect $P$. This contradicts uniqueness of passage times.
\end{proof}

As a consequence, we have the following global results on directionality. For the remainder of this section, we work on the event 
\begin{equation}\label{eq: X_def}
\mathcal{X},~ \text{defined by the following conditions:}
\end{equation}
\begin{enumerate}
\item for all $n$ and $x \in L_0$, there exist infinite geodesics $\Gamma_n^*$ and $\Gamma_n^{H,*}$ for $* = L,R$ from $x$ that are directed in $\Theta_n^*$ and $\Theta_n^{H,*}$ (the second sequence consists of half-plane infinite geodesics), and
\item all distinct finite paths have distinct passage times.
\end{enumerate}

\begin{cor}\label{cor: directed_sector}
For $\omega \in \mathcal{X}$, the following holds. Let $x_1, x_2, \ldots$ be any (random) sequence in $\mathbb{Z}^2$ which is asymptotically directed in $S$ with $x_n \to \infty$. Then for each $x \in \mathbb{Z}^2$, any subsequential limit of the finite geodesics $(\Gamma(x,x_n))$ is asymptotically directed in $S$. If additionally each $x_i$ and $x$ is in $V_H$, then any subsequential limit of $(\Gamma_H(x,x_n))$ is asymptotically directed in $S$.
\end{cor}
\begin{proof}
For simplicity, we give the proof in the case $x=0$; otherwise, similar arguments apply. Let $\Gamma$ be a subsequential limit of $(\Gamma(x,x_n))$ and write the vertices of $\Gamma$ in order as $0, y_1, y_2, \ldots$. Suppose that $(\mathrm{arg}~y_n)$ has a limit point in $S^c$ with distance $\epsilon$ to $S$. We first consider the setting of $\Z^2$. See Figure~\ref{fig: fig_2} for a depiction of the following argument.

\begin{figure}[h]
\setlength{\unitlength}{.5in}
\begin{picture}(10,3)(-5.5,2.5)
\linethickness{3pt}
\put(-4,2.5){\line(1,0){10}}
\linethickness{1pt}
\put(1,2.5){\line(0,1){.5}}
\put(1,3){\circle*{.2}}
\put(.9,3.2){$z$}
\put(1,3){\line(1,0){1}}
\put(2,3){\line(0,1){.5}}
\put(2,3.5){\line(1,0){.5}}
\put(2.5,3.5){\line(0,1){1.5}}
\put(1.7,4){$\Gamma^R$}
\put(2.5,5){\line(1,0){1}}
\put(3.5,5){\line(0,1){.5}}

\put(1,3){\line(-1,0){.5}}
\put(.5,3){\line(0,1){2}}
\put(0,4){$\Gamma^L$}
\put(.5,5){\line(1,0){1}}
\put(1.5,5){\line(0,1){.5}}

\linethickness{2pt}
\put(1,2.5){\line(0,1){.25}}
\put(1,2.75){\line(1,0){2}}
\put(3,2.75){\line(0,1){.25}}
\put(3,3){\line(1,0){.5}}
\put(3.5,3){\circle*{.2}}
\put(3.7,3){$y_k$}
\put(3.5,3){\line(0,1){1.5}}
\put(3.5,4.5){\line(-1,0){.5}}
\put(3,4.5){\line(0,1){.75}}
\put(3,5.25){\circle*{.2}}
\put(2.5,5.25){$x_n$}
\linethickness{1pt}
\put(3.5,3.5){\line(1,0){1}}
\put(4.5,3.5){\line(0,1){1}}
\put(4.5,4.5){\line(1,0){.5}}
\put(5,4.5){\line(0,1){1}}
\put(5,4.2){$\Gamma$}

\end{picture}
\caption{Illustration of the argument of Corollary~\ref{cor: directed_sector}. The infinite geodesics $\Gamma^L$ and $\Gamma^R$ are directed in $\Theta^L$ and $\Theta^R$. $R$ is the region between them and contains all but finitely many $x_n$. $\Gamma$ is subsequential limit of the $\Gamma(0,x_n)$'s and contains a $y_k$ which is not in $R$ or its boundary. The finite geodesic $\Gamma(0,x_n)$ (in bold) contains $y_k$ but must then cross into $R$ to touch $x_n$.}
\label{fig: fig_2}
\end{figure}
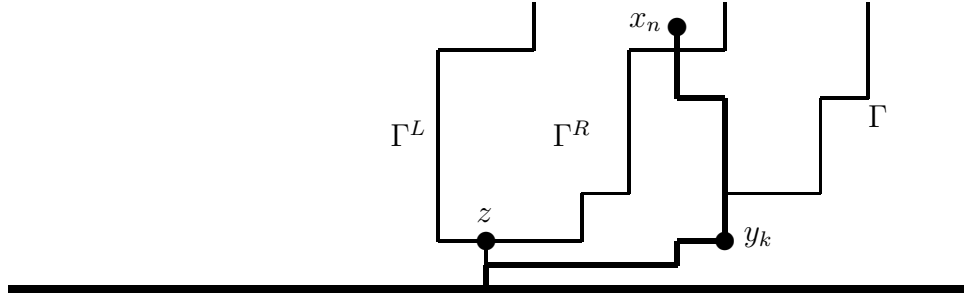

By definition of $\mathcal{X}$, we may find geodesics $\Gamma^L$ and $\Gamma^R$ directed in sectors $\Theta^L \subset (\theta_2, \pi)$ and $\Theta^R \subset (0,\theta_1)$ such that each angle in $\Theta^L \cup \Theta^R$ is at most distance $\epsilon/2$ away from $S$. The paths $\Gamma^L$ and $\Gamma^R$ split at some vertex $z$; let $P$ be the simple path formed by the union of $\Gamma^L$ from $z$ onward and $\Gamma^R$ from $z$ onward. Let $R$ be the component of $\mathbb{R}^2 \setminus P$ which contains all but finitely many of $x_1, x_2, \ldots$. Then there must be a vertex $y_k$ contained in the other component of $\mathbb{R}^2 \setminus P$ denoted $\hat R$. Then we can choose $n$ large enough that $\Gamma(0,x_n)$ contains both $x_n$, a vertex of $R$, and $y_k$, a vertex of $\hat R$. This contradicts uniqueness of passage times.

Now consider the case of $\mathbb{H}$. By definition of $\mathcal{X}$, we may find $\Gamma^L_H$ and $\Gamma_R^H$ starting at $0$, directed in sectors $\Theta^L_H \subseteq (\theta_2, \pi)$ and $\Theta_H^R \subseteq (0, \theta_1)$ such that each angle in $\Theta^L_H \cup \Theta_H^R$ is at most distance $\epsilon/2$ away from $S$. From here, the proof is identical to the preceding case.

\end{proof}


A statement of the same feel as Corollary~\ref{cor: directed_sector} is below. We will need it to ensure that certain subpaths of infinite geodesics stay in the upper half-plane. Specifically, it will control the number of ``backtracks'' that a geodesic can make into the lower half-plane. Let $L_n = L_0 + ne_2 = \{(x,n) : x \in \mathbb{Z}\}$.

\begin{cor}\label{cor: out_tree}
There exists $\epsilon > 0$ such that the following holds with probability one for each $k \geq 0$. For all large $n$,
\[
\left(\bigcup_{y \in L_\epsilon(n)}\mathrm{out}_0(y)\right) \cap L_k = \emptyset,
\]
where $\mathrm{out}_z(w)$ is the set of vertices $u$ in $\mathbb{Z}^2$ which have $w \in \Gamma(z,u)$ and
\[
L_\epsilon(n) = \{x \in L_n : \mathrm{arg}~x \in [\theta_1-\epsilon, \theta_2+\epsilon]\}.
\]
An identical statement holds replacing $\mathrm{out}_0(w)$ with $\out_0^H(w):= \{u \in \Z^2: \, w \in \Gamma_H(0,u)\}$.
\end{cor}
\begin{proof}
See Figure~\ref{fig: fig_3} for an illustration of the following argument. First consider $\Z^2$. From definition of $\mathcal{X}$, choose $\epsilon > 0$ such that there are sectors $S^L \subset (\theta_2+\epsilon, \pi)$ and $S^R \subset (0,\theta_1-\epsilon)$ such that with probability one, there are infinite geodesics $\Gamma^R$ and $\Gamma^L$ starting at $0$ and directed in $S^R$ and $S^L$ respectively. $\Gamma^R$ has a last intersection $x^R$ with $L_k$ and $\Gamma^L$ has a last intersection $x^L$ with $L_k$. Note that the portions $\Gamma_1$ of $\Gamma^L$ and $\Gamma_2$ of $\Gamma^R$ from $x^L$ and $x^R$ onward lie on or above $L_k$ and $x^L$ is not strictly to the right of $x^R$. Then $\Gamma_1$ and $\Gamma_2$ split this shifted half-plane into three regions: $R_1$ to the left of $\Gamma_1$, $R_2$ between $\Gamma_1$ and $\Gamma_2$, and $R_3$ to the right of $\Gamma_3$.

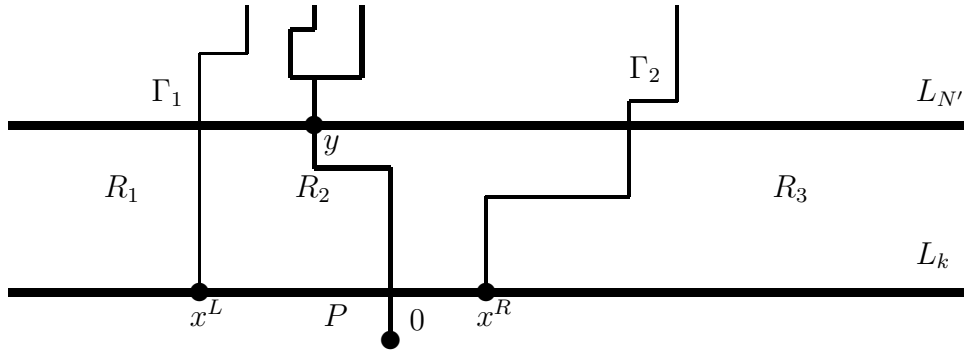
\begin{figure}[h]
\setlength{\unitlength}{.5in}
\begin{picture}(10,3.5)(-5.5,2)
\linethickness{3pt}
\put(-4,2.5){\line(1,0){10}}
\linethickness{3pt}
\put(-4,4.25){\line(1,0){10}}
\linethickness{1pt}
\put(-3,3.5){$R_1$}
\put(-1,3.5){$R_2$}
\put(4,3.5){$R_3$}
\put(-2.5,4.5){$\Gamma_1$}
\put(2.5,4.75){$\Gamma_2$}
\put(-2.1,2.15){$x^L$}
\put(.9,2.15){$x^R$}
\put(-.7,2.15){$P$}
\put(5.5,2.8){$L_k$}
\put(5.5,4.5){$L_{N'}$}
\put(-.7,4){$y$}
\put(.2,2.1){$0$}
\put(-2,2.5){\circle*{.2}}
\put(1,2.5){\circle*{.2}}
\put(0,2){\circle*{.2}}
\linethickness{1.5pt}
\put(0,2.5){\line(0,-1){.5}}
\linethickness{1pt}
\put(-2,2.5){\line(0,1){2.5}}
\put(-2,5){\line(1,0){.5}}
\put(-1.5,5){\line(0,1){.5}}
\put(1,2.5){\line(0,1){1}}
\put(1,3.5){\line(1,0){1.5}}
\put(2.5,3.5){\line(0,1){1}}
\put(2.5,4.5){\line(1,0){.5}}
\put(3,4.5){\line(0,1){1}}
\linethickness{1.5pt}
\put(-.8,4.25){\circle*{.2}}
\put(-.8,4.25){\line(0,1){.5}}
\put(-.8,4.75){\line(1,0){.5}}
\put(-.8,4.75){\line(-1,0){.25}}
\put(-.3,4.75){\line(0,1){.75}}
\put(-1.05,4.75){\line(0,1){.5}}
\put(-1.05,5.25){\line(1,0){.25}}
\put(-.8,5.25){\line(0,1){.25}}
\put(0,2.5){\line(0,1){1.3}}
\put(0,3.8){\line(-1,0){.8}}
\put(-.8,3.8){\line(0,1){.45}}
\end{picture}
\caption{Depiction of the argument of Corollary~\ref{cor: directed_sector}. The infinite geodesics $\Gamma_1$ and $\Gamma_2$ and directed in sectors $S^L \subset (\theta_2+\epsilon,\pi)$ and $S^R \subset (0,\theta_1-\epsilon)$. Their last intersections with $L_k$ are $x^L$ and $x^R$. The regions composing the complement of $\Gamma_1$ and $\Gamma_2$ in the upper half-plane are $R_1$, $R_2$, and $R_3$. The portion of $L_k$ between $x^L$ and $x^R$ is denoted $P$ (it could be empty). The intersection of $L_{N'}$ with $R_2$ is contained in $L_\epsilon(N')$; that is, the points have angle within $\epsilon$ of $S_\theta$. $N'$ is chosen so large that no geodesic from 0 to a vertex of $P$ intersects $L_{N'}$.}
\label{fig: fig_3}
\end{figure}

The region $R_2$ has as boundary the curves $\Gamma_1$, $\Gamma_2$, and possibly a line segment of $L_k$ called $P$ between $x^L$ and $x^R$. ($P$ might be empty if $x^L = x^R$.) Define $\mathcal{C}$ to be the union of all vertices in any geodesic between $0$ and a vertex of $P$. Note that by uniqueness of passage times, $\mathcal{C}$ is finite, so we can define $N$ to be the maximal $e_2$-coordinate of any vertex in $\mathcal{C}$. Choose a further $N' \geq N$ such that for $n \geq N'$, the set $L_\epsilon(n)$ is contained in $R_2$. Then each $\mathrm{out}_0(y)$ for $y \in L_\epsilon(n)$ must be contained in $R_2$, for if not, a geodesic from 0 must touch $y$ and then intersect either $\Gamma_1$ or $\Gamma_2$ (and contradict uniqueness of passage times) or $P$, which is impossible since $n \geq N'$.

Now consider $\mathbb{H}$. $\mathcal{X}$ again gives the existence of geodesics $\Gamma_H^L, \, \Gamma_H^R$ starting at $0$ directed in sectors $\Theta^L \subseteq (\theta_2, \pi)$ and $\Theta^R \subseteq (0, \theta_1)$. As in the case of $\Z^2$ above, the union of the segments of $\Gamma^L_H$ and $\Gamma_H^R$ after their last intersection forms a simple path which splits the shifted half-plane into three disjoint regions (note that in this case we do not need to use a segment of the $e_1$-axis as part of the curve). The remainder of the argument proceeds identically.
\end{proof}


\subsection{Definition of ordering}

We continue considering only $\omega \in \mathcal{X}$, defined in \eqref{eq: X_def}. Let $\mathcal{G}(x)$ (respectively $\mathcal{G}_H(x)$) be the set of infinite geodesics from $x \in \mathbb{Z}^2$ (respectively half-plane infinite geodesics from $x \in V_H$) that are directed in $S$. We define an ordering on both 
\[
\mathcal{G} := \bigcup_{x \in \mathbb{Z}^2} \mathcal{G}(x) \text{ and } \mathcal{G}_H := \bigcup_{x \in V_H} \mathcal{G}_H(x)
\]
as follows. Two infinite geodesics $\Gamma, \Gamma'$ directed in $S$ are said to be ordered as $\Gamma \prec \Gamma'$ if $\Gamma'$ is ``asymptotically to the left'' of $\Gamma$. That is, $\Gamma \prec \Gamma'$ if for all large $n$, the left-most intersection of $\Gamma'$ with $L_n$ occurs to the left of or at the left-most intersection of $\Gamma$ with $L_n$.
\begin{lem}
\label{lem:totalorderedpizza}
The relation $\prec$ defines a total ordering on each of $\mathcal{G}$ and $\mathcal{G}_H$. That is, the following statements hold for $\Gamma, \Gamma',$ and $\Gamma''$ infinite geodesics directed in $S$.
\begin{enumerate}
\item $\Gamma \prec \Gamma'$ or $\Gamma' \prec \Gamma$,
\item If $\Gamma \prec \Gamma'$ and $\Gamma' \prec \Gamma$, then $\Gamma$ and $\Gamma'$ coalesce, and
\item If $\Gamma \prec \Gamma'$ and $\Gamma' \prec \Gamma''$, then $\Gamma \prec \Gamma''$.
\end{enumerate}
\end{lem}
Strictly speaking, these conditions define a total order on equivalence classes of geodesics, where two geodesics are identified if they coalesce.
\begin{proof}
Statement 3 is obvious, so we prove only 1 and 2. It suffices to prove that if $\Gamma$ and $\Gamma'$ do not coalelsce, then $\Gamma \prec \Gamma'$ or $\Gamma' \prec \Gamma$ (but not both). By uniqueness of passage times, the fact that $\Gamma$ and $\Gamma'$ do not coalesce implies that they intersect only finitely often. Since they are directed in $S$, we can find $N$ such that the following hold: (a) the last intersection $x$ of $\Gamma$ with $L_N$ is (say) strictly to the left of the last intersection $x'$ of $\Gamma'$ with $L_N$ and (b) the portion of $\Gamma$ from $x$ onward and the portion of $\Gamma'$ from $x'$ onward do not intersect and they lie in the half-space $\cup_{n \geq N} L_n$.

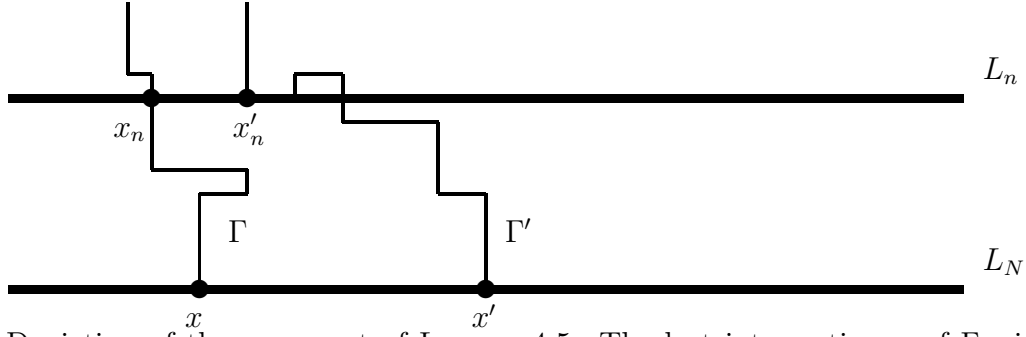
\begin{figure}[h]
\setlength{\unitlength}{.5in}
\begin{picture}(10,3)(-5.5,2.5)
\linethickness{3pt}
\put(-4,2.5){\line(1,0){10}}
\put(6.2,2.7){$L_{N}$}
\put(-4,4.5){\line(1,0){10}}
\put(6.2,4.7){$L_n$}
\linethickness{1pt}
\put(1,2.5){\circle*{.2}}
\put(.85,2.1){$x'$}
\put(1,2.5){\line(0,1){1}}
\put(1.2,3){$\Gamma'$}
\put(1,3.5){\line(-1,0){.5}}
\put(.5,3.5){\line(0,1){.75}}
\put(.5,4.25){\line(-1,0){1}}
\put(-.5,4.25){\line(0,1){.5}}
\put(-.5,4.75){\line(-1,0){.5}}
\put(-1,4.75){\line(0,-1){.25}}
\put(-1,4.5){\line(-1,0){.5}}
\put(-1.5,4.5){\line(0,1){1}}
\put(-1.5,4.5){\circle*{.2}}
\put(-1.65,4.1){$x_n'$}
\put(-2,2.5){\circle*{.2}}
\put(-2.15,2.1){$x$}
\put(-2,2.5){\line(0,1){1}}
\put(-1.7,3){$\Gamma$}
\put(-2,3.5){\line(1,0){.5}}
\put(-1.5,3.5){\line(0,1){.25}}
\put(-1.5,3.75){\line(-1,0){1}}
\put(-2.5,3.75){\line(0,1){1}}
\put(-2.5,4.5){\circle*{.2}}
\put(-2.9,4.1){$x_n$}
\put(-2.5,4.75){\line(-1,0){.25}}
\put(-2.75,4.75){\line(0,1){.75}}
\end{picture}
\caption{Depiction of the argument of Lemma~\ref{lem:totalorderedpizza}. The last intersection $x$ of $\Gamma$ with $L_{N}$ is strictly to the left of the last intersection $x'$ of $\Gamma'$ with $L_{N}$. The path $P'$ consists of the portion of $\Gamma'$ from $x'$ to $x'_n$ and the part of $L_n$ to the left of $x_n'$. The path $\Gamma$ from $x$ onward must cross $P$ and must do so to the left of $x_n'$.}
\label{fig: fig_4}
\end{figure}

See Figure~\ref{fig: fig_4} for an illustration of the following argument. For $n \geq N$, let $x_n$ (respectively $x_n'$) be the left-most intersection of $\Gamma$ (respectively $\Gamma'$) with the line $L_n$. Choose $N' \geq N$ such that for $n \geq N'$, the portions of $\Gamma$ from $x$ to $x_n$ and $\Gamma'$ from $x'$ to $x_n'$ lie in $\cup_{m > N} L_m$ except for the initial point. We claim that for $n \geq N'$, $x_n$ is strictly to the left of $x_n'$. To see why, let $P'$ be the path consisting of the union of the portion of $\Gamma'$ from $x'$ to $x_n'$ with the portion of $L_n$ to the left of $x_n'$. Now the portion of $\Gamma$ from $x+(1/2)e_2$ onward begins in the region between $P'$ and the $L_N$, but it must reach infinity in the sector $S$. Therefore it must cross $P'$, but it cannot touch $\Gamma'$, so it intersects $L_n$ strictly to the left of $x_n'$. Since this is true for all $n \geq N'$, $\Gamma \prec \Gamma'$.
\end{proof}

\begin{remark}\label{rem: first_intersection}
An equivalent characterization of the ordering is as follows. One has $\Gamma \prec \Gamma'$ if and only if for all large $n$, the first intersection of $\Gamma'$ with $L_n$ is to the left or at the first intersection of $\Gamma$ with $L_n$. The proof is similar to the above.
\end{remark}

\begin{remark}\label{rem: cannot_remove}
A similar proof also shows the following, which we will use later. Let $\Gamma$ and $\Gamma'$ be infinite geodesics in $\mathbb{Z}^2$ that are directed in $S$ and start at the same vertex in $L_0$. Suppose that for some $1 \leq n<n'$, the following hold:
\begin{enumerate}
\item the left-most intersection of $\Gamma$ with $L_n$ is strictly to the left of the left-most intersection of $\Gamma'$ with $L_n$ and
\item the left-most intersection of $\Gamma'$ with $L_{n'}$ is strictly to the left of the left-most intersection of $\Gamma$ with $L_{n'}$.
\end{enumerate}
Then one of the following two must be true: (a) the portion of $\Gamma$ from its left-most intersection with $L_n$ to its left-most intersection with $L_{n'}$ touches $L_0$ or (b) the portion of $\Gamma'$ from its left-most intersection with $L_n$ to its left-most intersection with $L_{n'}$ touches $L_0$.
\end{remark}

\begin{remark}\label{rem: opposite_order}
If one defines an ordering $\prec_R$ on geodesics directed in $S$ by $\Gamma \prec_R \Gamma'$ if for all large $n$, the right-most intersection of $\Gamma$ with $L_n$ is equal or to the right of the right-most intersection of $\Gamma'$ with $L_n$, then $\Gamma \prec_R \Gamma'$ if and only if $\Gamma' \prec \Gamma$.
\end{remark}

The claim in Remark~\ref{rem: opposite_order} is proved as follows. Suppose that $\Gamma \prec_R \Gamma'$ but that $\Gamma \neq \Gamma'$. By replacing these paths with the portions from the last intersection with $L_0$, we may assume they lie entirely in the upper half-plane. Now choose $n$ large so that the right-most intersection $w_n^R$ of $\Gamma$ with $L_n$ is strictly to the right of the right-most intersection of $\Gamma'$ with $L_n$. If the left-most intersection $w_n^L$ of $\Gamma$ with $L_n$ is strictly to the left of the left-most intersection of $\Gamma'$ with $L_n$, then we build a curve $P$ by concatenating the portion of $L_n$ to the left of $w_n^L$ with the portion of $\Gamma$ from $w_n^L$ to $w_n^R$ and the portion of $L_n$ to the right of $w_n^R$. Then $P$ does not intersect $\Gamma'$ but separates the $e_1$-axis from infinity in the upper half-plane, and this is a contradiction.

\subsection{Existence of extremal geodesics}

Again, we work with $\omega$ in $\mathcal{X}$ from \eqref{eq: X_def}. For $n \geq 0$, let $V_x(n)$ (respectively $V_{x,H}(n)$) be the set of vertices in $L_n$ which are also in some element of $\mathcal{G}(x)$ (respectively $\mathcal{G}_H(x)$). Let $v_x^L(n)$ and $v_x^R(n)$ (respectively $v_{x,H}^L(n)$ and $v_{x,H}^R(n)$) be the left-most and right-most points  of $V_x(n)$ (respectively $V_{x,H}(n)$).

\begin{lem}
\label{lem:v_finite}
The sets $(V_0(n): n \geq 1)$ are almost surely finite. Moreover, any sequence $(x_n)$ with $x_n \in V_0(n)$ is directed in $S$.
\end{lem}
\begin{proof}

We make use of particular geodesics from the definition of $\mathcal{X}$. Find $\varepsilon > 0$, sectors $S^L \subset (\theta_2 + \epsilon, \pi)$ and $S^R \subseteq (0, \theta_1 - \epsilon)$, and infinite geodesics $\Gamma^L, \, \Gamma^R$ starting at $0$, directed in $S^L$ and $S^R$ (respectively). Note that $\Gamma^R$ and $\Gamma^L$ have some last intersection point $p$; let $P$ denote the doubly infinite simple curve formed by concatenating the segments of $\Gamma^R$ and $\Gamma^L$ beginning at $p$.

See Figure~\ref{fig: fig_5} for an illustration of the following argument. Assume for the sake of contradiction that there is some $n \geq 1$ such that $|V_0(n)|$ is infinite. In particular, if we let $x$ (respectively $y$) denote the leftmost (respectively rightmost) vertex of $(\Gamma^L \cup \Gamma^R) \cap L_n$ and $[x,y]$ the segment of $L_n$ between them, then there is some $z \notin [x, y]$ and an infinite geodesic $\Gamma \ni z$ directed in $S$ and starting at $0$. Writing $\Gamma = (0 = z_0, z_1, \ldots)$, note that $\Gamma$ is identical with $\Gamma^L$ ($\Gamma^R$) up to some first branching point $z_{k^L}$ ($z_{k^R}$). Letting $k = \max\{k^L, \, k^R\}$, note that by construction $z$ appears in $\Gamma$ after $z_k$. By uniqueness of geodesics, the segment $\Gamma'$ of $\Gamma$ beginning at $z$ cannot intersect $P$.

\begin{figure}[h]
\setlength{\unitlength}{.5in}
\begin{picture}(10,6)(-5.5,2.5)
\linethickness{3pt}
\put(-4,3.5){\line(1,0){10}}
\put(-4,5.5){\line(1,0){10}}
\put(-4,6.25){\line(1,0){10}}
\linethickness{2pt}
\put(-3,2.5){\line(0,1){5}}
\put(-3,7.5){\line(1,0){8}}
\put(5,7.5){\line(0,-1){5}}
\linethickness{1pt}
\put(0,3.5){\circle*{.2}}
\put(-.4,3.1){$0$}
\put(6.2,3.5){$L_0$}
\put(6.2,5.5){$L_n$}
\put(6.2,6.25){$L_m$}
\put(5,7.8){$\partial [-K,K]^2$}
\put(-.5,5){$\Gamma^L$}
\put(1.6,5){$\Gamma^R$}
\put(1.5,5.5){\circle*{.2}}
\put(1.7,5.7){$y$}
\put(.25,5.5){\circle*{.2}}
\put(0,5.7){$x$}
\put(2.5,5.5){\circle*{.2}}
\put(2.7, 5.7){$z$}
\put(0,4){\circle*{.2}}
\put(-.5,4){$p$}
\put(-.25,7.5){\circle*{.2}}
\put(-.6,7.1){$y'$}
\put(1.25,7.5){\circle*{.2}}
\put(1.45,7.1){$x'$}
\put(.5,6.25){\circle*{.2}}
\put(.1,6.4){$z_\ell$}
\put(0,3.5){\line(0,1){.5}}
\put(0,4){\line(1,0){.5}}
\put(.5,4){\line(0,1){1}}
\put(.5,5){\line(1,0){1}}
\put(1.5,5){\line(0,1){1}}
\put(1.5,6){\line(-1,0){.5}}
\put(1,6){\line(0,-1){.75}}
\put(1,5.25){\line(-1,0){.25}}
\put(.75,5.25){\line(0,1){2}}
\put(.75,7.25){\line(1,0){.5}}
\put(1.25,7.25){\line(0,1){.75}}
\put(0,4){\line(0,1){1}}
\put(0,5){\line(1,0){.25}}
\put(.25,5){\line(0,1){1}}
\put(.25,6){\line(-1,0){.5}}
\put(-.25,6){\line(0,1){2}}
\put(2.5,5.5){\line(0,1){1}}
\put(2.5,6.5){\line(1,0){.5}}
\put(3,6.5){\line(0,1){.75}}
\put(3,7.25){\line(-1,0){.5}}
\put(2.5,7.25){\line(0,1){.75}}
\put(3.2,7){$\Gamma'$}
\end{picture}
\caption{Depiction of the argument of Lemma~\ref{lem:v_finite}. The infinite geodesics $\Gamma^L$ and $\Gamma^R$ start at 0, split at $p$, and are directed in $S^L$ and $S^R$. The left and rightmost intersections of their union with $L_n$ are $x$ and $y$. $z$ is a vertex of $V_0(n)$ outside of the interval $[x,y]$, and $\Gamma'$ is the portion from $z$ of and infinite geodesic $\Gamma$ starting from $0$ that contains $z$. $\Gamma'$ must touch the vertex $z_\ell$, which is in the region between $\Gamma^L$ and $\Gamma^R$, but since it cannot cross the other geodesics, it must leave the box before doing so.}
\label{fig: fig_5}
\end{figure}
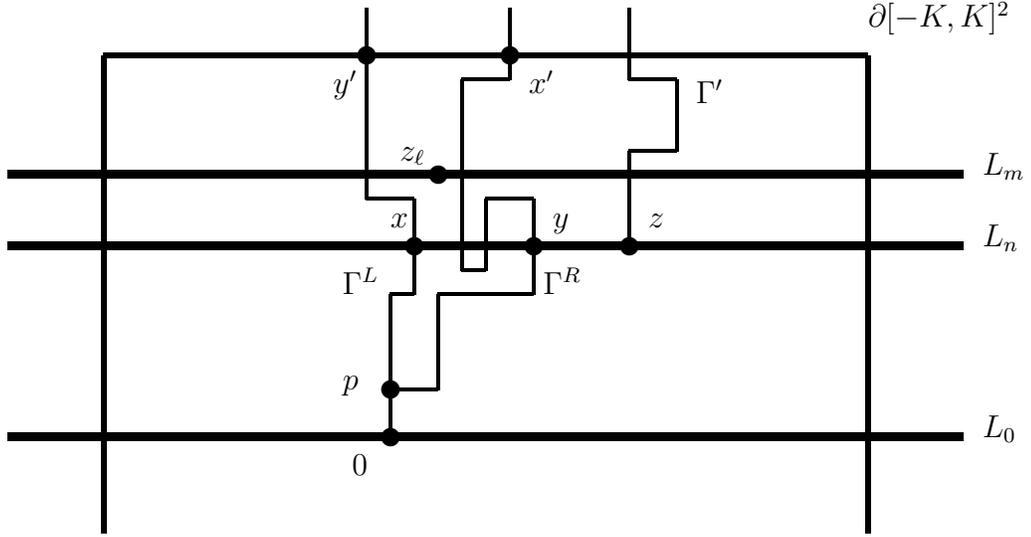

Choose $M > \|p\|_\infty + \|z\|_\infty$ such that, for all $m > M$, each $x^L \in \Gamma^L$ (respectively $x^R \in \Gamma^R$) lying on the boundary $\partial [-m, m]^2$ of $[-m, m]^2$ has $\arg (x^L) \in (\theta_2 + \epsilon, \pi)$ (respectively $\arg (x^R) \in (0, \theta_1 - \epsilon)$). In particular, for all $m > M$ and all $x \in L_m$ with $\arg x \in (\theta_1 - \epsilon / 2, \theta_2 + \epsilon/2)$, the ray
$\{x + \alpha e_2: \, \alpha \geq 0\}$ is disjoint from $P$.

Since $\Gamma'$ is directed in $S$, there is some $m > M$ and $z_\ell \in L_m \cap \Gamma'$ with $\arg z_\ell \in (\theta_1 - \epsilon / 2, \theta_2 + \epsilon/2)$. We will see that this is a contradiction, as $\Gamma'$ is blocked from reaching $z_\ell$. This will follow once we show that for all $K > M$, the path $\Gamma'$ exits $[-K, K]^2$ before reaching $z_\ell$.

Fix such $K$ and construct a Jordan $P'$ from $P$ as follows. Let $x', \, y'$ denote the first intersections of the two ends of the path $P$ with $\partial [-K, K]^2$, where $\arg x' \in (0, \theta_1 - \epsilon)$ and $\arg y' \in (\theta_2 + \epsilon, \pi)$. Then $P'$ is defined as the concatenation of the segment of $P$ from  $x'$ to $y'$ with the clockwise path along $\partial [-K, K]^2$ from $y'$ to $x'$. $P'$ divides $\mathbb{R}^2$ into two components. By the choice of $M$, $z_\ell$ lies in the bounded component, as the ray $\{z_\ell + \alpha e_2:\, \alpha \geq 0\}$ crosses $P'$ exactly once along the top side of $[-K,K]^2$.

On the other hand, $z$ lies in the unbounded component, since $\{z + \alpha (z-x): \, \alpha \geq 0\}$ does not intersect $P'$. In particular, since it connects vertices in distinct components, $\Gamma'$ must intersect $P'$ before $z_\ell$.  On the other hand, $\Gamma'$ cannot intersect $P$, so this intersection point must happen on $\partial [-K, K]^2$.

The proof of the second claim is similar, so we will only specify the differences in the argument. Assume there exists some $\epsilon > 0$ and a sequence $(x_{n_k})$ with $x_{n_k} \in V_0(n_k)$ and $\arg(x_{n_k}) \notin (\theta_1 - 2\epsilon, \theta_2 + 2\epsilon)$ for all $k$. Using the definition of $\mathcal{X}$ as before, we can find a $\Gamma^R$ asymptotically directed in some sector $S^R \subseteq (\theta_1 - \epsilon, \theta_1)$ and $\Gamma^L$ in $S^L \subseteq (\theta_2, \theta_2+\epsilon)$.

We now choose some large $k$ and a $z \in V_0(n_k)$ lying either to the left or to the right of the intersections of $\Gamma^L,\, \Gamma^R$ with $L_{n_k}$, analogously to the previous case. Some geodesic $\Gamma$ starting at $0$ must intersect $z$, then return to hit infinitely many vertices $y$ with 
\[\arg y \in (\theta_1 - \epsilon, \theta_2 + \epsilon) \setminus (S^R \cup S^L)\ .\]
But by taking $M$ large and considering arbitrary $K > M$, we see as before that $\Gamma$ must exit $[-K,K]^2$ before hitting $z$, a contradiction.

\end{proof}

In the next proposition, we construct the left- and right-most infinite geodesics directed in $S$. Note that, although the geodesics from definition of $\mathcal{X}$ may not be $(t_e)$-measurable, those constructed below are.
\begin{prop}\label{prop: left-most_definition}
Almost surely, for each $x \in \mathbb{Z}^2$, the limits $\Gamma_x^L = \lim_n \Gamma(x,v_x^L(n))$ and $\Gamma_x^R = \lim_n \Gamma(x,v_x^R(n))$ exist and are directed in $S$. These geodesics satisfy the following properties.
\begin{enumerate}
\item $\Gamma_x^L$ is the unique maximal element of $\mathcal{G}(x)$ under the ordering $\prec$. $\Gamma_x^R$ is similarly minimal.
\item For $* = L,R$, $\Gamma_x^*$ contains all but finitely many of $(v_x^*(n) : n \geq 1)$.
\item For $* = L,R$, if $y \in \mathbb{Z}^2$ is in $\Gamma_x^*$, then $\Gamma_y^*$ is the portion of $\Gamma_x^*$ beginning at $y$.
\item $\Gamma_x^L$ can be characterized as follows. Let $(v_k)$ be any sequence tending to infinity and directed in $S$ such that $v_k \in L_{n_k}$ and $v_k$ is to the left or at $v_0^L(n_k)$. Then $\lim_k \Gamma(0,v_k) = \Gamma_0^L$. A similar statement holds for $\Gamma_x^R$.
\end{enumerate}

Similar statements hold for half-plane limits and the limiting geodesics.
\end{prop}

\begin{proof}
Since the half-plane case is similar to (and easier than) the full-plane case, we concentrate on the latter.

\emph{Existence}.
This style of argument will be used a few times, so we provide the full details here for $\Gamma_x^L$. Without loss of generality, we can assume that $x = 0$. We first show the existence of the limiting geodesic $\Gamma^L_0$. Note that there exists at least one subsequential limiting geodesic $\Gamma = \lim_k \Gamma(0, v_0^L(n_k))$. Moreover, by Lemma \ref{lem:v_finite}, $v_0^L(n_k)$ is directed in $S$, so by Corollary \ref{cor: directed_sector}, $\Gamma$ is directed in $S$. 

Assume that there were another subsequential limit $\Gamma' = \lim_\ell \Gamma(0, v_0^L(m_\ell)) \neq \Gamma$, which must also be directed in $S$. By Lemma \ref{lem:totalorderedpizza}, we may assume without loss of generality that $\Gamma~\prec~\Gamma'$. Fix some $N$ such that, for all $n \geq N$, the following hold.
\begin{enumerate}
\item the leftmost point of $\Gamma' \cap L_n$ is strictly to the left of the leftmost point of $\Gamma \cap L_n$ and, 
by Remark~\ref{rem: first_intersection}, the first intersection of $\Gamma'$ with $L_n$ is strictly to the left of the first intersection of $\Gamma$ with $L_n$. Furthermore, $(\Gamma \cap L_n) \cap (\Gamma' \cap L_n) = \emptyset$.
\item $\out_0(y) \cap L_0 = \emptyset$ for all $y \in L_{\epsilon}(n)$  (by Corollary~\ref{cor: out_tree}, from which the value of $\epsilon<\pi/4$ is taken).
\item $V_0(n) \subseteq L_\epsilon(n)$, so $(\Gamma \cap L_n) \cup (\Gamma' \cap L_n) \subseteq L_\epsilon (n)$.
\item Each of $\Gamma, \, \Gamma'$ only intersect $ \partial [-n, n]^2$ in the top side.
\end{enumerate}

Choose $K$ large enough that if $k \geq K$, then $n_k > N$ and that the geodesic $\Gamma(0, v_0^L(n_k))$ is identical with $\Gamma$ up to the first intersection $z$ of $\Gamma$ with $L_{N+1}$. Last, choose $N' > n_K$ such that if $n \geq N'$, then $\Gamma'$ and $\Gamma$ have their last intersections with $L_{N+1}$ before leaving $[-n, n]^2$. For each $n > N'$, consider the Jordan curve $P$ formed by the segment of $\Gamma'$ from its last intersection $y$ with $L_0$ to its first intersection $y'$ with $L_n$, the segment of the $e_1$ axis from the right side of $y$ to the boundary of $[-n, n]^2$, the segment of $L_n$ from the right of $y'$ to the boundary of $[-n, n]^2$, and the right side of $[-n, n]^2$.

\begin{figure}[h]
\setlength{\unitlength}{.5in}
\begin{picture}(10,6)(-5.5,2.5)
\linethickness{3pt}
\put(-4,3.5){\line(1,0){10}}
\put(-4,5.5){\line(1,0){10}}
\put(-4,6){\line(1,0){10}}
\put(-4,6.5){\line(1,0){10}}
\linethickness{2pt}
\put(-3,2.5){\line(0,1){5}}
\put(-3,7.5){\line(1,0){8}}
\put(5,7.5){\line(0,-1){5}}
\linethickness{1pt}
\put(0,3.5){\circle*{.2}}
\put(2,5.5){\circle*{.2}}
\put(-2,6){\circle*{.2}}
\put(-2.2,5.65){$v_0^L(n_K)$}
\put(2.1,5.2){$z$}
\put(-.4,3.1){$0$}
\put(0,3.5){\line(0,-1){.5}}
\put(0,3){\line(1,0){1}}
\put(1,3){\line(0,1){1}}
\put(1,4){\line(1,0){.25}}
\put(1.25,4){\line(0,-1){.75}}
\put(1.5,3.5){\circle*{.2}}
\put(1.7,3.2){$y$}
\put(1.7,4.2){$\Gamma'$}
\put(1.7, 7.8){$y'$}
\put(6.2,3.5){$L_0$}
\put(6.2,5.5){$L_{N+1}$}
\put(6.2,6){$L_{n_K}$}
\put(6.2,6.5){$L_{N'}$}
\put(5,7.8){$\partial [-n,n]^2$}
\put(1.25,3.25){\line(1,0){.25}}
\put(1.5,3.25){\line(0,1){1.75}}
\put(1.5,5){\line(-1,0){.5}}
\put(1,5){\line(0,1){1.25}}
\put(1,6.25){\line(1,0){.5}}
\put(1.5,6.25){\line(0,1){2}}
\put(1.5,7.5){\circle*{.2}}
\put(2,5.5){\line(0,1){1.5}}
\put(2,7){\line(1,0){.5}}
\put(2.5,7){\line(0,1){1.25}}
\end{picture}
\caption{Illustration of the argument of existence in Proposition~\ref{prop: left-most_definition}. The Jordan curve $P$ consists of the portion of $\Gamma'$ from $y$ to $y'$, which are the last intersection of $\Gamma'$ with $L_0$ and the first with $L_n$, and the portions of $\partial [-n,n]^2$ clockwise from $y'$ to $y$. The point $z$ is the first intersection of $\Gamma$ with $L_{N+1}$ and, following $\Gamma$ from this point, we first intersect $L_n$ strictly to the right of $y'$. The leftmost point $v_0^L(n_K)$ of $V_0^L(n_K)$ is in the unbounded component of the complement of $P$ and $\Gamma(0,v_0^L(n_K))$, which agrees with $\Gamma$ up to $z$, must leave the bounded component of the complement of $P$ from $z$ to reach it.}
\label{fig: fig_6}
\end{figure}
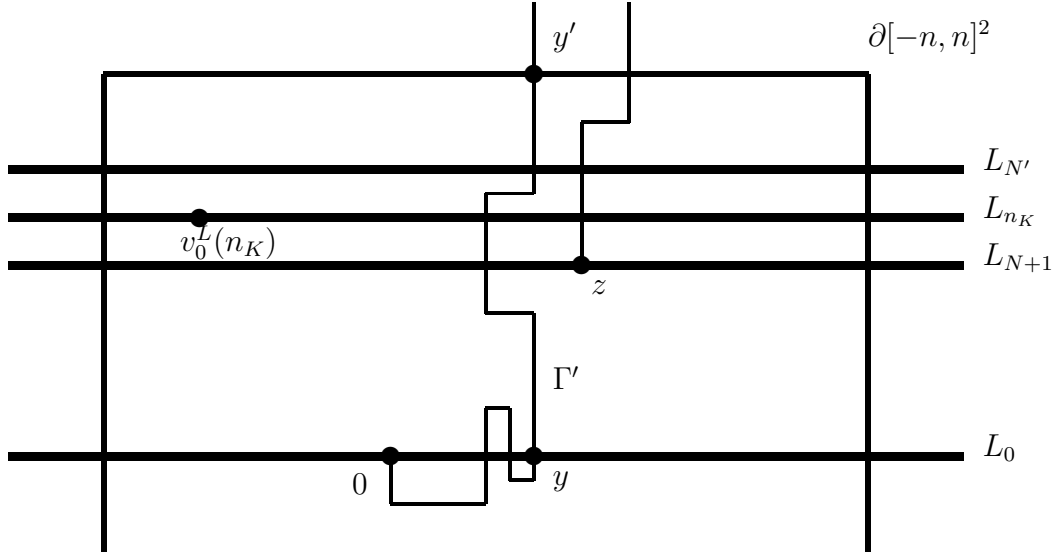

We claim that $z$ is in the interior of $P$. Indeed, traversing the segment of $\Gamma$ forwards from $z$, we never intersect $\Gamma'$ by assumption. So there is a first intersection with the boundary of $[-n, n]^2$ (along the top side), which must be to the right of $y'$ by item 1 above, so it is a point of $P$. This is necessarily the first intersection with $P$ by item 2. From here, continuing our trajectory by moving straight in the positive $e_2$ direction, we never re-intersect $P$, proving we have passed into the exterior of $P$.

On the other hand, $v_0^L(n_K)$ is in the exterior of $P$. To see why, first it cannot be in $\partial [-n,n]^2$. Next if it were in $\Gamma'$, then we note that $z$ is not in $\Gamma'$ by item 1, but $z$ is in $\Gamma(0,v_0^L(n_K))$, so this would contradict uniqueness of geodesics. Since $\Gamma(z, v_0^L(n_K))$ starts in the interior of $P$, it must intersect $P$ before $v_0^L(n_K)$. But this intersection occurs on $L_n$ by uniqueness of geodesics. Since this is true for all $n$ large, we see that $v_0^L(n_K) \notin \Gamma(z, v_0^L(n_K))$, a contradiction.

\emph{Property 1}. Assume there were some $\Gamma \neq \Gamma_0^L$ with $\Gamma_0^L \prec \Gamma$. A construction identical to that of the existence proof, with $\Gamma$ here filling the role of $\Gamma'$ from that argument, shows that $\Gamma_0^L$ will necessarily be isolated from infinitely many leftmost vertices, a contradiction.

\emph{Property 2}. Assume for the sake of contradiction that there were infinitely many $v_0^L(n)$'s not lying on $\Gamma_0^L$. By Lemma \ref{lem:v_finite}, the sequence $v_0^L(n)$ is asymptotically contained in $L_\epsilon(n)$, where $\epsilon$ is as in Corollary~\ref{cor: out_tree}. So for infinitely many $n$, we can find $\Gamma \neq \Gamma_0^L$ with $\Gamma \ni v_0^L(n)$ but $\Gamma \prec \Gamma_0^L$. This implies by Remark \ref{rem: cannot_remove} that $\Gamma$ or $\Gamma_0^L$ must intersect $L_0$ after intersecting $L_n$. In particular, for infinitely many large $n$ we can find $y \in L_\epsilon(n)$ such that $\out_0(y) \cap L_0 \neq \emptyset$, a contradiction.

\emph{Property 4}. Let $\Gamma$ be a subsequential limit of $\Gamma(0,v_k)$ and note that it is directed in $S$ since $(v_k)$ is. If it is not equal to $\Gamma_0^L$, then by property 1, $\Gamma_0^L \prec \Gamma$. Now we use the same argument as in the existence proof. To summarize, one chooses $N$ so large that the leftmost vertex $w_N$ of $\Gamma$ in $L_N$ is strictly to the right of $\Gamma_0^L$ and $\text{out}_0(w_N)$ does not touch $L_0$. Then one takes $k$ so large that $\Gamma(0,v_k)$ agrees with $\Gamma$ until $w_N$. This contradicts uniqueness of passage times, as $\Gamma(0,v_k)$ must pass from $w_N$, which is to the right of $\Gamma_0^L$, to $v_k$, which is to the left of $\Gamma_0^L$.

\emph{Property 3}. Assume for the sake of contradiction that for some $y \in \Gamma_0^L$, $\Gamma_y^L$ is not identical to the terminal segment of $\Gamma_0^L$. By the preceding work, we see that $\Gamma_y^L$ is directed in $S$, so either $\Gamma_y^L \prec \Gamma_0^L$ or $\Gamma_0^L \prec \Gamma_y^L$. Because subsegments of geodesics are geodesics, the terminal segment of $\Gamma_0^L$ is a geodesic from $y$ directed in $S$; therefore, $\Gamma_0^L \prec \Gamma_y^L$. Let $\Gamma_y^L = (y=z_1, z_2, \ldots)$.

By property 4, $\lim_k \Gamma(0,z_k) = \Gamma_0^L$. Once again, choose $k$ large enough so that $\Gamma(0,z_k)$ agrees with $\Gamma$ up to and one step beyond the point where $\Gamma$ and $\Gamma_y^L$ split. This contradicts uniqueness of geodesics, as the segment of $\Gamma(0,z_k)$ from $y$ to $z_k$ is a geodesic, but so is the portion of $\Gamma_y^L$ from $y$ to $z_k$, and they are different.
\end{proof}

\subsection{Coalescence of extremal geodesics}

We begin with coalescence for extremal full-plane geodesics using the argument of Licea-Newman. Since this has appeared so many times in the literature (in directed LPP geometric weights \cite{FP05}, general weights \cite{GRS2}, the Hammersley process \cite{CP11}, directed spanning forest \cite{CT12}, and so on), we just give a sketch, but prove in detail one fact which is needed to use the argument for extremal geodesics.
\begin{prop}\label{prop: full_plane_coalescence}
With probability one, for each $x,y \in \mathbb{Z}^2$, $\Gamma_x^L$ coalesces with $\Gamma_y^L$ and $\Gamma_x^R$ coalesces with $\Gamma_y^R$.
\end{prop}
\begin{proof}
The argument is simpler with unbounded weights, and an extension to weights which obey assumptions {\bf A1'} or {\bf A2'} appeared in \cite[Section~6]{DH}. So for ease of exposition, we assume i.i.d. weights that are unbounded. Assume that with positive probability, $\Gamma_0^L$ and $\Gamma_x^L$ do not coalesce for some $x \in \mathbb{Z}^2$. Then using the fact that these paths are directed in $S$ and the ergodic theorem, we can find $x_1,x_2 \in L_0$ such that $x_1$ is to the left of 0, $x_2$ is to the right of zero, and with positive probability, the paths $\Gamma_0^L, \Gamma_{x_1}^L,$ and $\Gamma_{x_2}^L$ (a) intersect $L_0$ only at their starting points and remain in the upper half-plane and (b) do not share any vertices. On this event, one then uses either the i.i.d. assumption or the upward finite energy condition to ``modify'' all edges connecting the interval $[x_1,x_2]$ to $L_{-1}$ to be so large that no geodesic in $\mathbb{Z}^2$ uses these edges. This forces an event of positive probability on which no infinite geodesics starting in the lower half-plane can intersect $\Gamma_0^L$, due to the high-weight barrier. One then completes the proof with a Burton-Keane argument, which shows that in a large box, there is a volume-order number of such ``isolated'' infinite geodesics, but each must intersect the boundary of this box in a unique point, and this is a contradiction since the size of the boundary of the box is smaller order than the volume.

To carry out this argument in our context, we need to know that when we increase the edge-weights of the barrier, the left-most infinite geodesics of $0, x_1,x_2$ do not change. A key point is that the edges we modify are not in these geodesics. Thus we need to verify:
\begin{lem}
With probability one, the following holds. Let $f$ be an edge and for $r \geq 0$, denote by $t(r)$ the configuration which agrees with $(t_e)$ except at $f$, where it takes the value $r$. Then if $f$ is not in the geodesic $\Gamma_0^L$ in configuration $(t_e)$ and $r \geq t_f$, one has $\Gamma_0^L((t_e)) = \Gamma_0^L(t(r))$.
\end{lem}
\begin{proof}
Denote by $\Gamma_0^L(r)$ the left-most infinite geodesic from $0$ in $S$ in the configuration $t(r)$ and assume that it does not equal $\Gamma_0^L$. First of all, $\Gamma_0^L$ is still an infinite geodesic directed in $S$ in the weights $t(r)$, so if $\Gamma_0^L(r) \neq \Gamma_0^L$, we must have $\Gamma_0^L(r) \prec \Gamma_0^L$ strictly. Next, $f$ cannot be contained in $\Gamma_0^L(r)$, or else this path would be a geodesic in the original weights $(t_e)$ and this contradicts the extremality of $\Gamma_0^L$. By property 4 of Proposition~\ref{prop: left-most_definition}, one has $\Gamma_0^L = \lim_k \Gamma(0,z_k)$, where $z_1, z_2, \ldots$ are the vertices of $\Gamma_0^L(r)$. But now we choose $N$ so large that $v_0^L(N)$ is not contained in $\Gamma_0^L(r)$ and $\text{out}_0(v_0^L(N)) \cap (L_0 \cup f) = \emptyset$ (by Corollary~\ref{cor: out_tree}). Let $k$ be so large that $\Gamma(0,z_k)$ agrees with $\Gamma_0^L$ up to $v_0^L(N)$. Then $\Gamma(0,z_k)$ contains $v_0^L(N)$ and since it is a geodesic in the original weights $(t_e)$, it must remain one in $t(r)$ (as $r \geq t_f$ and it does not contain $f$). This contradicts uniqueness of passage times, as $\Gamma(0,z_k)$ and the portion of $\Gamma_0^L(r)$ from $0$ to $z_k$ must have the same passage time in the original weights.
\end{proof}

\end{proof}




We wish to show that for all $x,y \in L_0$, the geodesics $\Gamma_{x,H}^L$ and $\Gamma_{y,H}^L$ coalesce. To do this, we first show:
\begin{prop}\label{prop: equality}
On the positive probability event that $\Gamma_0^L$ has all its vertices in $V_H$, almost surely $\Gamma_0^L = \Gamma_{0,H}^L$.  A similar statement holds for right-most geodesics.
\end{prop}
\begin{proof}
We prove only the first; the second is similar. Almost surely, the geodesic $\Gamma_0^L$ is directed in $S$. Since $\theta_1 > 0$ and $\theta_2 < \pi$, there must be a last intersection point $w$ of $\Gamma_0^L$ with $L_0$. Then $\Gamma_w^L$ is the portion of $\Gamma_0^L$ from $w$ onward, so $\Gamma_w^L$ has only vertices in $V_H$. By horizontal translation invariance, $\mathbb{P}(A_0)>0$, where $A_0$ is the event
\[
A_0 = \{\Gamma_0^L \text{ has only vertices in } V_H\}.
\]
For $\omega \in A_0$, the path $\Gamma_0^L$ is an infinite geodesic in $\mathbb{H}$, so if we denote by $v_n$ the left-most intersection point of $\Gamma_0^L$ with $L_n$, then $v_n = v_0^L(n)$ for all large $n$ and so the point $v_{0,H}^L(n)$ cannot be to the right of $v_n$. However, since $\Gamma_{0,H}^L$ contains all but finitely many vertices $v_{0,H}^L(n)$ and is directed in $S$ (by Proposition~\ref{prop: left-most_definition}), the sequence $(v_{0,H}^L(n))$ is directed in $S$. So property 4 of Proposition~\ref{prop: left-most_definition} implies that $\lim_n \Gamma(0,v_{0,H}^L(n)) = \Gamma_0^L$.

If $v_n = v_{0,H}^L(n)$ for infinitely many $n$, then $\Gamma_0^L = \Gamma_{0,H}^L$, so assume otherwise. Then there is $N = N(\omega)$ such that if $n \geq N$, then $v_{0,H}^L(n)$ is strictly to the left of $v_n$ (and therefore $v_{0,H}^L(n) \notin \Gamma_0^L$). Further increase $N$ so that (as $(v_n)$ is directed in $S$), by Corollary~\ref{cor: out_tree}, the set $\text{out}_0(v_N)$ does not intersect $L_0$. Now find $n$ large enough so that $\Gamma(0,v_{0,H}^L(n))$ agrees with $\Gamma_0^L$ up to $v_N$. Then $\Gamma(0,v_{0,H}^L(n))$ is in the upper half-plane, but contains $v_N$, which is not a vertex on $\Gamma_{0,H}^L$, and this last path begins with the unique geodesic from 0 to $v_{0,H}^L(n)$ in the upper half-plane. This is a contradiction to uniqueness of passage times.
\end{proof}

Now we can give the coalescence result.
\begin{prop}
Almost surely, for all $x,y \in V_H$, $\Gamma_{x,H}^L$ coalesces with $\Gamma_{y,H}^L$ and $\Gamma_{x,H}^R$ coalesces with $\Gamma_{y,H}^R$.
\end{prop}
\begin{proof}
Again, we prove only for left-most geodesics. Let $x,y \in L_0$. By the ergodic theorem, one can find $z_1, z_2 \in L_0$ such that $z_1$ is to the left of both $x$ and $y$, and $z_2$ is to the right of both $x$ and $y$, such that $\Gamma_{z_1}^L = \Gamma_{z_1,H}^L$ and $\Gamma_{z_2}^L = \Gamma_{z_2,H}^L$. Since both $\Gamma_{z_1}^L$ and $\Gamma_{z_2}^L$ coalesce, so do $\Gamma_{z_1,H}^L$ and $\Gamma_{z_2,H}^L$. By planarity then, $\Gamma_{x,H}^L$ must intersect one of $\Gamma_{z_1,H}^L$ or $\Gamma_{z_2,H}^L$. However, item 3 from Proposition~\ref{prop: left-most_definition} implies that $\Gamma_{x,H}^L$ coalesces with both $\Gamma_{z_1,H}^L$ and $\Gamma_{z_2,H}^L$. The same is true for $\Gamma_{y,H}^L$, so they must coalesce together.

A similar trapping argument gives coalescence for $x,y \in V_H$. Pick $N_1 > \max\{x\cdot e_2, y \cdot e_2\}$ and $N_2 > \max\{|x \cdot e_1|, |y \cdot e_1|\}$. Since $\Gamma_{0,H}^L$ is directed in $S$, it has finitely many intersections with $D := \cup_{k=0}^{N_1} L_k$. So pick $b>0$ such that with positive probability, $\Gamma_{0,H}^L$ intersects $D$ only in the box $[-b,b]^2$. By the ergodic theorem, there are infinitely many $z \in L_0$ (in both directions from $0$) such that $\Gamma_{z,H}^L$ intersects $D$ only in $[-b,b]^2 + z$. Choose such a $z$ on the left of $-N_2 - b$ and such a $z'$ on the right of $N_2 + b$. Then $\Gamma_{z,H}^L$ and $\Gamma_{z',H}^L$ coalesce and trap both $\Gamma_{x,H}^L$ and $\Gamma_{y,H}^L$.
\end{proof}

\section{Equality of left and right-most geodesics}\label{sec: equal_geodesics}

Here we prove asymptotics for Busemann differences. Due to the last section, we can define the half-plane Busemann function
\[
B_H^L(x,y) = \lim_n \left[ T_H(x,y_n) - T_H(y,y_n) \right] \text{ for } x,y \in V_H,
\]
where the vertices of $\Gamma_{0,H}^L$ are $0,y_1, y_2, \ldots$, in order. This limit exists by monotonicity for $y=y_0$:
\begin{align*}
T_H(x,y_n) - T_H(y_0,y_n) &= T_H(x,y_n)+ T_H(y_n,y_{n+1}) - T_H(y_0,y_{n+1}) \\
&\geq T_H(x,y_{n+1})-T_H(y,y_{n+1})
\end{align*}
and for general $y$ we use $B_H^L(x,y) = B_H^L(x,y_0) + B_H^L(y_0,y)$. A similar definition is made for $B_H^R$. Note that by coalescence, we can choose $\Gamma_{z,H}^L$ instead of $\Gamma_{0,H}^L$ for any $z \in V_H$.

Now we define
\[
\Delta_H(x,y) = B_H^L(x,y) - B_H^R(x,y) \text{ for } x,y \in V_H.
\]
\begin{prop}\label{prop: bigger}
If $\mathbb{P}(\Gamma_{0,H}^L \neq \Gamma_{0,H}^R) > 0$, then $\mathbb{E}\Delta_H(0,e_1)<0$.
\end{prop}
\begin{proof}
We first show that $\Delta_H(0,e_1) \leq 0$ almost surely. For this, we claim that $\Gamma_{0,H}^R$ and $\Gamma_{e_1,H}^L$ share a vertex. To see why, let $n$ be such that $\Gamma_{e_1,H}^L$ contains $v_{0,H}^L(n)$, the left-most vertex of $V_{0,H}^n$. Such an $n$ exists because $\Gamma_{e_1,H}^L$ and $\Gamma_{0,H}^L$ coalesce. Consider the continuous path $P$ formed by following $\Gamma_{e_1,H}^L$ until $v_{0,H}^L(n)$, and then proceeding to the left on $L_n$. Since $\Gamma_{0,H}^R$ is directed in $S$, it must cross $P$. It cannot cross $L_n$ to the left of $v_{0,H}^L(n)$, so it must cross $\Gamma_{e_1,H}^L$ at some vertex $z$.

Let $y$ be a vertex on $\Gamma_{e_1,H}^L$ beyond $z$ and let $y'$ be a vertex on $\Gamma_{0,H}^R$ beyond $z$. Then
\begin{align*}
T_H(0,y') + T_H(e_1,y) &= T_H(0,z) + T_H(z,y') + T_H(e_1,z) + T_H(z,y) \\
&\geq T_H(0,y) + T_H(e_1,y').
\end{align*}
Rearranging and taking $y,y' \to \infty$ along their respective geodesics, we obtain $B_H^L(0,e_1) \leq B_H^R(0,e_1)$, or $\Delta_H(0,e_1) \leq 0$.

Next we must prove that since $\mathbb{P}$ has uniqueness of passage times, $\Delta_H(0,e_1) \neq 0$ with positive probability. This will complete the proof. If it were false, then by additivity, $\Delta_H(0,ke_1) = 0$ for all $k$ almost surely. Choose an integer $b_1>0$ such that $c:=\mathbb{P}(A(b_1))>0$, where
\[
A(b_1) = \{(\Gamma_{0,H}^L \Delta \Gamma_{0,H}^R) \cap [-b_1,b_1]^2 \neq \emptyset\}.
\]
Next, pick an integer $b_2>0$ such that $\mathbb{P}(B(b_1,b_2)) > 1-c$, where
\[
B(b_1,b_2) = \{(\Gamma_{0,H}^L \cup \Gamma_{0,H}^R) \cap (\cup_{i=0}^{b_1} L_i) \subset [-b_2,b_2] \times [0,b_1]\}.
\]
Setting $S$ to be the translation by $e_1$, one has by translation invariance
\[
\mathbb{P}(A(b_1) \cap T^{4b_2} B(b_1,b_2)) > 0.
\]
On this event, there is an edge $e$ in $[-b_1,b_1]^2$ which is in $\Gamma_{0,H}^L$ but not in $\Gamma_{0,H}^R$. (It could be in $\Gamma_{0,H}^R$ but not in $\Gamma_{0,H}^L$, but the argument is similar.) This edge cannot be in $\Gamma_{4b_2e_1,H}^L$ due to occurrence of $T^{4b_2} B(b_1,b_2)$. We claim that this forces $\Delta_H(0,4b_2e_1) \neq 0$, which is a contradiction. Indeed, write $z_1$ for the point of coalescence for $\Gamma_{0,H}^L$ and $\Gamma_{4b_2e_1,H}^L$ and $z_2$ for the point of coalescence for $\Gamma_{0,H}^R$ and $\Gamma_{4b_2e_1,H}^R$. Then
\[
\Delta_H(0,4b_2e_1) = T_H(0,z_1) + T_H(4b_2e_1,z_2) - \left[ T_H(0,z_2) + T_H(4b_2e_1,z_1) \right].
\]
However only the geodesic between $0$ and $z_1$ contains $e$, so $\Delta_H(0,4b_2e_1) = 0$ contradicts uniqueness of passage times.
\end{proof}
The preceding proposition is where the restriction to half-planes is really needed. Without this restriction, the leftmost geodesic from $e_1$ need not touch the rightmost from $0$, so one could not conclude $\Delta(0,e_1) \leq 0$ almost surely. 

Next, we use differentiability of the limit shape to show
\begin{prop}\label{prop: delta_0}
One has $\mathbb{E}\Delta_H(0,e_1)=0$. Therefore $\Gamma_{0,H}^L = \Gamma_{0,H}^R$ almost surely.
\end{prop}
\begin{proof}
We start by showing a similar statement for the full-plane Busemann functions. We will show that 
\begin{equation}\label{eq: to_show_ast}
\mathbb{E}B^L (0,e_1) = \mathbb{E}B^R(0,e_1).
\end{equation}

The function $x \mapsto \mathbb{E}B^\ast(0,x)$ is linear in $x$, so there is a vector $\rho^\ast \in \mathbb{R}^2$ such that
\[
\mathbb{E}B^\ast(0,x) = \rho^\ast \cdot x \text{ for } x \in \mathbb{Z}^2.
\]
In fact, under our assumptions, one can replicate the proof of the shape theorem for Busemann functions \cite[Theorem~4.3]{DH} exactly to show that for each $\epsilon>0$,
\begin{equation}\label{eq: busemann_shape_thm}
\mathbb{P}\left( |B^\ast(0,x) - \rho^\ast \cdot x| > \epsilon\|x\|_1 \text{ for infinitely many }x \in \mathbb{Z}^2 \right)=0.
\end{equation}
Label the last intersections of $\Gamma_0^\ast$ with $L_0, L_1, L_2, \ldots$ as $x_0, x_1, x_2, \ldots$. Then choose a subsequence $(n_k)$ such that $x_{n_k}/g(x_{n_k})$ converges to some vector $z \in \partial \mathcal{B}$. By the shape theorem,
\[
B^\ast(0,x_{n_k})/\|x_{n_k}\|_1 = T(0,x_{n_k})/\|x_{n_k}\|_1 \to g(z) = 1.
\]
On the other hand, by \eqref{eq: busemann_shape_thm}, this must converge to $\rho^\ast \cdot z$. Therefore
\[
\rho^\ast \cdot z = 1.
\]
To complement this equation, one has for all $r \in \mathcal{B}$,
\[
\rho^\ast \cdot r = \lim_n B^\ast(0,[nr])/n \leq \lim_n T(0,[nr])/n = g(r) \leq 1,
\]
where $[nr]$ is the integer point such that $nr \in [nr] + [0,1)^2$. 

These two statements imply that the line
\[
L^* = \{r \in \mathbb{R}^2 : r \cdot \rho^\ast = 1\}
\]
is a supporting line for the limit shape at $z$; that is, each point of the limit shape lies on one side of the line, and the line touches the limit shape at $z$. Since the argument of $z$ lies in $S$, our assumed differentiability conditions imply that $L^*$ is equal to the tangent line to the limit shape at $v_\theta$. In particular, $\rho^L=\rho^R$, meaning that \eqref{eq: to_show_ast} holds.

To derive the same statements for the half-plane Busemann functions, note that by the ergodic theorem, $\lim_n B_H^*(0,ne_1)/n$ exists and is constant almost surely, so it suffices to show that $\liminf_n |B_H^\ast (0,ne_1) - \rho^\ast \cdot ne_1|/n = 0$ on the positive probability event $E$ that $\Gamma_0^\ast = \Gamma_{0,H}^\ast$. (Here we are using Proposition~\ref{prop: equality}.) By the ergodic theorem, we can find an infinite increasing sequence $(i_k : k \geq 1)$ of positive integers such that $\gimel^{i_k}E$ occurs, where $\gimel$ is translation by $e_1$. Then $B_H^\ast (0,i_ke_1) = B^\ast (0,i_ke_1)$ and again by the ergodic theorem,
\begin{align*}
\liminf_n \frac{|B_H^\ast(0,ne_1) - \rho^*\cdot ne_1|}{n} &\leq \liminf_k \frac{|B_H^\ast(0,i_k e_1) - \rho^* \cdot i_ke_1|}{i_k} \\
&= \liminf_k \frac{|B^\ast(0,i_ke_1) - \rho^* \cdot i_ke_1|}{i_k} = 0.
\end{align*}
\end{proof}

\section{Proofs of main theorems}

\begin{proof}[Proof of Theorem~\ref{thm: limits}]
We first show that with probability one, $\Gamma_0^L = \Gamma_0^R$ for $\theta \in [\pi/4,\pi/2]$. For $\ast = L$ and $R$, the geodesic $\Gamma_0^\ast$ has a last intersection $z^\ast$ with $L_0$. By Proposition~\ref{prop: equality}, $\Gamma_{z^\ast}^\ast = \Gamma_{z^\ast,H}^\ast$. But $\Gamma_{z^L,H}^L = \Gamma_{z^L,H}^R$ coalesces with $\Gamma_{z^R,H}^R = \Gamma_{z^R}^R$, so $\Gamma_0^L$ and $\Gamma_0^R$ coalesce, and this implies they are equal.

We now prove that there is a unique geodesic directed in $S$ starting from $0$. Due to Corollary~\ref{cor: directed_sector}, this will prove item 1. Define $\Gamma_0 = \Gamma_0^L = \Gamma_0^R$. Assume $\Gamma$ is any geodesic directed in $S$ starting from $0$ which is not equal to $\Gamma_0$. Since $\Gamma_0$ is maximal and minimal under the ordering $\prec$, we find that $\Gamma_0 \prec \Gamma$ and $\Gamma \prec \Gamma_0$, so $\Gamma = \Gamma_0$.

Item 2 holds because $\Gamma_x^L$ coalesces with $\Gamma_y^L$. For item 3, build the directed geodesic graph $\mathbb{G} = \mathbb{G}(S)$ in sector $S$ as follows. The vertices of $\mathbb{G}$ are the vertices of $\mathbb{Z}^2$. A directed edge $\langle x,y \rangle$ is in $\mathbb{G}$ if $\{x,y\}$ is in $\Gamma_x$. Then $\mathbb{G}$ satisfies the conditions of \cite[Remark~6.10]{DH}, so for each $x \in \mathbb{Z}^2$, one has $\#C_x < \infty$ with probability one, where
\[
C_x = \{y \in \mathbb{Z}^2 : x \in \Gamma_y\}.
\]
Now, if there is a bigeodesic with an end directed in $S$ with positive probability, then there must exist a vertex $x$ such that $x$ is in such a bigeodesic with positive probability. Then there are infinitely many vertices $y$ such that one end of this bigeodesic starting from $y$ contains $x$ and is directed in $S$. However, this end is an infinite geodesic directed in $S$, so it must equal $\Gamma_y$, and therefore $\#C_x = \infty$. This event has probability zero.
\end{proof}

\begin{proof}[Proof of Theorem~\ref{thm: busemann}]
Let $(x_n)$ be a sequence of points in $\mathbb{Z}^2$ which is directed in $S$. For $x,y \in \mathbb{Z}^2$, the sequences of finite geodesics $(\Gamma(x,x_n))$ and $(\Gamma(y,x_n))$ converge to $\Gamma_x$ and $\Gamma_y$, which coalesce at a point $z$, so we can find $N$ such that for $n \geq N$, both $\Gamma(x,x_n)$ and $\Gamma(y,x_n)$ contain $z$. Then for $n \geq N$, $T(x,x_n) - T(y,x_n) = T(x,z) - T(y,z)$, so the limit in $n$ exists.

To show the second statement, we can use exactly the same arguments as in the previous section (paragraph of \eqref{eq: busemann_shape_thm}). To recall, we note that $x \mapsto \mathbb{E}B(0,x)$ is linear and so there is a $\rho \in \mathbb{R}^2$ such that $\mathbb{E}B(0,x) = \rho \cdot x$ for all $x \in \mathbb{Z}^2$. Now a shape theorem-type argument gives for each $\epsilon>0$
\[
\mathbb{P}(|B(0,x) - \rho \cdot x| > \epsilon \|x\|_1 \text{ for infinitely many }x \in \mathbb{Z}^2) = 0.
\]
Again putting $z_1, z_2, \ldots$ as the first intersections of $\Gamma_0^L$ with $L_1, L_2, \ldots$, choose a subsequence $(z_{n_k})$ so that $z_{n_k}/g(z_{n_k})$ converges to a vector $z$ which has direction in $S$. Then as before, $\rho \cdot z = 1$ and for all $r \in \mathcal{B}$, one has $\rho \cdot r \leq 1$. Under our differentiability assumption, this means that $\{r : \rho \cdot r = 1\}$ is the supporting line for the limit shape at $v_\theta$.
\end{proof}

\bigskip
\noindent
{\bf Acknowledgements.} We thank D. Ahlberg and two anonymous referees for many comments on the manuscript.


\begin{thebibliography}{1}

\bibitem{A98}
S. E. Alm. (1998). A note on a problem by Welsh in first-passage percolation. {\it Combin. Probab. Comput.} {\bf 7}, 11-15.

\bibitem{AW99}
S. E. Alm and J. Wierman. (1999). Inequalities for means of restricted first-passage times in percolation theory. {\it Combin. Probab. Comput.} {\bf 8}, 307-315.

\bibitem{survey}
A. Auffinger, M. Damron, and J. Hanson. (2015). 50 years of first-passage percolation. {\it arXiv:1511.03262}.

\bibitem{BCK14}
Y. Bakhtin, E. Cator, and K. Khanin. (2014). Space-time stationary solutions for the Burgers equation. {\it J. Amer. Math. Soc.} 27, 193-238.

\bibitem{BKS}
I. Benjamini, G. Kalai and O. Schramm. First passage percolation has sublinear distance variance. {\it Ann. Probab.} {\bf 31}, 1970--1978.

\bibitem{boivin}
D. Boivin. (1990). First passage percolation: the stationary case. {\it Probab. Theory Relat. Fields.} {\bf 86}, 491--499.

\bibitem{BD02}
D. Boivin and J.-M. Derrien. (2002). Geodesics and recurrence of random walks in disordered systems. {\it Electron. Comm. Probab.} 7, 101-115.

\bibitem{CP11}
E. Cator and L. P. R. Pimentel. (2011). A shape theorem and semi-infinite geodesics for the Hammersley model with random weights. {\it ALEA}. 8, 163-175.

\bibitem{C11}
D. Coupier. (2011). Multiple geodesics with the same direction. {\it Electron. Commun. Probab.} {\bf 16}, 517-527.

\bibitem{CT12}
D. Coupier and V. Tran. (2012). The $2D$-directed spanning forest is almost surely a tree. {\it Random Struct. Alg.} 42, 59-72.

\bibitem{DH}
M. Damron and J. Hanson. (2014). Busemann functions and infinite geodesics in two-dimensional first-passage percolation. {\it Commun. Math. Phys.} {\bf 325}, 917--963.

\bibitem{FP05}
P. Ferrari and L. P. R. Pimentel. (2005). Competition interfaces and second class particles. {\it Ann. Probab.} 33, 1235-1254.

\bibitem{FLN}
G. Forgacs, R. Lipowsky, and T. M. Nieuwenhuizen. The behaviour of interfaces in ordered and disordered systems, pp. 135-363. in {\it Phase Transitions and Critical Phenomena,} C. Domb and J. Lebowitz (eds.), Vol. 14, Acad. Press, London, 1991.


\bibitem{GRS} N. Georgiou, F. Rassoul-Agha and T. Sepp\"{a}l\"{a}inen. (2015). Stationary cocycles and Busemann functions for the corner growth model. {\it arXiv:1510.00859.}

\bibitem{GRS2} N. Georgiou, F. Rassoul-Agha and T. Sepp\"al\"ainen. (2015). Geodesics and the competition interface for the corner growth model. {\it arXiv:1510.00860}.

\bibitem{HM} 
O. H\"aggstr\"om and  R. Meester. (1995). Asymptotic shapes for stationary first passage percolation. {\it Ann. Probab.} {\bf 23}, 1511--1522. 

\bibitem{Hoffman05}
C. Hoffman. (2005). Coexistence for Richardson type competing spatial growth models. {\it Ann. Appl. Probab.} {\bf 15}, 739-747.

\bibitem{Hoffman08}
C. Hoffman. (2008). Geodesics in first-passage percolation. {\it Ann. Appl. Probab.} {\bf 18}, 1944--1969.

\bibitem{HN97}
C. D. Howard and C. M. Newman. (1997). {\it Euclidean models of first-passage percolation.} {\bf 108}, 153-170.

\bibitem{HN01}
C. D. Howard and C. M. Newman. (2001). {\it Geodesics and spanning trees for Euclidean first-passage percolation.} {\it Probab. Theory Relat. Fields.} {\bf 29}, 577-623.

\bibitem{aspects}
H. Kesten. (1986). Aspects of first passage percolation. {\it \'Ecole d'\'et\'e de probabilit\'es de Saint-Flour, XIV--1984}, Lecture Notes in Math., 1180, {\it Springer, Berlin.}


\bibitem{LN96}
C. Licea and C. M. Newman. (1996). Geodesics in two-dimensional first-passage percolation. {\it Ann. Probab.} {\bf 24}, 399--410.

\bibitem{N95}
C. Newman. (1994). A surface view of first-passage percolation. {\it Proceedings of the International Congress of Mathematicians}, Z\"urich.

\bibitem{N97}
C. Newman. Topics in disordered systems. Lectures in Mathematics ETH Z\"urich. {\it Birkh\"auser Verlag, Basel,} 1997. viii+88 pp. ISBN: 3-7643-5777-0.

\bibitem{P06}
L. P. R. Pimentel. (2007). Multitype shape theorems for FPP models. {\it Adv. in Appl. Probab.} {\bf 39}, 53--76.

\bibitem{Tasaki}
H. Tasaki. (1989). On the upper critical dimensions of random spin systems. {\it J. Stat. Phys.} 54, 163-170.

\bibitem{Wehr}
J. Wehr. (1997). On the number of infinite geodesics and ground states in disordered systems. {\it J. Statist. Phys.} {\bf 87}, 439--447.

\bibitem{WW98}
J. Wehr and J. Woo. (1998). Absence of geodesics in first-passage percolation on a half-plane. {\it Ann. Probab.} {\bf 26}, 358--367.

\end{thebibliography}
\end{document}